\documentclass[11pt,a4paper]{amsart}

\usepackage{amsmath, amsthm, verbatim, amsfonts, amssymb, xcolor}
\usepackage{graphics, xspace, enumerate,bbm}
\usepackage[a4paper,margin=3cm]{geometry}
\usepackage{marginnote}

\usepackage[bb=boondox]{mathalfa}

\usepackage{graphicx}
\usepackage[a4paper,colorlinks=true,citecolor=green,urlcolor=green,linkcolor=green,bookmarksopen=true,unicode=true,pdffitwindow=true]{hyperref}

\hypersetup{pdfauthor={}}
\hypersetup{pdftitle={On Sub-Geometric Ergodicity of  Diffusion Processes}}

\theoremstyle{plain}
\newtheorem{theorem}{Theorem}[section]
\newtheorem{corollary}[theorem]{Corollary}
\newtheorem{lemma}[theorem]{Lemma}
\newtheorem{proposition}[theorem]{Proposition}
\theoremstyle{definition}
\newtheorem{remark}[theorem]{Remark}
\newtheorem{example}[theorem]{Example}

\newcommand {\D} {{\rm d}}
\newcommand {\E} {{\rm e}}
\newcommand {\Prob} {\ensuremath{\mathbb{P}}}
\newcommand {\R} {\ensuremath{\mathbb{R}}}
\newcommand {\ZZ} {\ensuremath{\mathbb{Z}}}
\newcommand {\N} {\ensuremath{\mathbb{N}}}
\newcommand {\V} {\ensuremath{\mathcal{V}}}
\newcommand {\W} {\ensuremath{\mathcal{W}}}

\newcommand{\process}[1]{\{#1_t\}_{t\geq0}}
\newcommand {\Ind} {\ensuremath{\mathbb{1}}}

\numberwithin{equation}{section}

\begin{document}
\allowdisplaybreaks[4]

\title{On Sub-Geometric Ergodicity of  Diffusion Processes}

\author[Petra\ Lazi\'{c}]{Petra Lazi\'{c}}
\address[Petra\ Lazi\'{c}]{Department of Mathematics\\University of Zagreb\\10000 Zagreb\\Croatia}
\email{petralaz@math.hr}

\author[Nikola\ Sandri\'{c}]{Nikola Sandri\'{c}}
\address[Nikola\ Sandri\'{c}]{Department of Mathematics\\University of Zagreb\\10000 Zagreb\\Croatia}
\email{nsandric@math.hr}

\subjclass[2010]{60J25, 60J75, 60G17}
\keywords{asymptotic flatness, diffusion process, sub-geometric ergodicity, total variation distance, Wasserstein distance}

\begin{abstract}
   In this article, we discuss  ergodicity properties of a diffusion process given through an It\^{o} stochastic differential equation. 
   We identify conditions on the drift and diffusion coefficients 
   which result
   in sub-geometric  ergodicity of the corresponding semigroup with respect to the total variation distance.  We also prove  sub-geometric contractivity and ergodicity of the  semigroup under a class of
   Wasserstein distances. Finally, we  discuss sub-geometric ergodicity of two classes of Markov processes with jumps. 
\end{abstract}

\maketitle

%
%
%
%

\section{Introduction}

One of the classical directions in the analysis   of Markov processes centers around  their ergodicity properties. In this article, we focus on both   qualitative and quantitative  aspects of this problem. More precisely, 
we discuss sub-geometric ergodicity  of a diffusion process given by 
\begin{equation}\label{eq1}\D X^x_t\,=\,b(X^x_t)\D t+\sigma(X^x_t)\D B_t\,,\qquad X^x_0\,=\,x\in\R^{d}\,,\end{equation} 
with respect to the total variation distance and/or a class of Wasserstein distances.
Here, $\process{B}$ stands for a standard $n$-dimensional Brownian motion (defined on a stochastic basis $(\Omega,\mathcal{F},\{\mathcal{F}_t\}_{t\ge0},\mathbb{P})$ satisfying the usual conditions), and the coefficients $b:\R^d\to\R^d$ and $\sigma:\R^d\to\R^{d\times n}$ satisfy: 
\begin{description}
	\item[(C1)] for any $r>0$,
	$$\sup_{x\in B_r(0)}(\rvert b(x)\rvert+\lVert\sigma(x)\lVert_{{\rm HS}})\,<\,\infty\,;$$
	\item[(C2)] for any $r>0$ there is $\Gamma_r>0$ such that for all $x,y\in B_r(0)$,  $$2\langle x-y,b(x)-b(y)\rangle+\lVert\sigma(x)-\sigma(y)\lVert_{{\rm HS}}^2\,\le\, \Gamma_r|x-y|^{2}\,;$$
	\item[(C3)] there is $\Gamma>0$ such that for all $x\in\R^d$,  $$2\langle x,b(x)\rangle+\lVert\sigma(x)\lVert_{{\rm HS}}^{2}\,\le\, \Gamma(1+|x|^2)\,,$$
\end{description}
where  $B_r(x)$ denotes the open ball with radius $r>0$ around $x\in\R^d$, and  $\lVert M\lVert_{{\rm HS}}^2:={\rm Tr}\,MM^T$ is the  Hilbert-Schmidt norm of a real  matrix $M.$
\subsection{Structural properties of the model} It is well known that under $\textbf{(C1)}$-$\textbf{(C3)}$, for any $x\in\R^d$, the stochastic differential equation (SDE) in \eqref{eq1} admits a unique strong non-explosive solution $\process{X^x}$ which is a strong Markov process with continuous sample paths and  transition kernel $p(t,x,\D y)=\Prob(X^x_t\in \D y)$, $t\ge0$, $x\in\R^d$,  
(see  \cite[Theorems 5.4.1, 5.4.5 and 5.4.6]{Durrett-Book-1996} and \cite[Theorem 3.1.1]{Prevot-Rockner-Book-2007}). In the context of Markov processes, it is natural that the underlying probability measure depends on the initial conditions of the process. Using standard arguments (Kolmogorov extension theorem), it is well known that for each $x\in\R^d$ the above defined transition kernel defines a unique probability measure $\mathbb{P}^x$ on the canonical (sample-path)  space such that the projection process, denoted by $\process{X}$, is a strong Markov process (with respect to the completion of the corresponding natural filtration), it has continuous  sample paths, 
and  the same finite-dimensional distributions (with respect to $\mathbb{P}^x$) as $\process{X^x}$ (with respect to $\mathbb{P}$). Since we are interested in distributional properties of the solution to   \eqref{eq1} only, in the sequel we rather deal with $\process{X}$ than with $\process{X^x}$.
According to \cite[Lemma 2.5]{Majka-2016}, $\process{X}$ is also a $C_b$-Feller process, that is, the corresponding semigroup, defined by  $$P_tf(x)\,:=\,\mathbb{E}^x[f(X_t)]=\int_{\R^d}f(y)p(t,x,\D y)\,,\qquad t\geq0\,, \ x\in\R^d\,,\ f\in B_b(\R^d)\,,$$ satisfies $P_t(C_b(\R^d))\subseteq C_b(\R^d)$.  Here, $B_b(\R^d)$ and $C_b(\R^d)$  denote the spaces of  bounded Borel measurable functions and bounded continuous functions, respectively. Let us remark that in the above-mentioned lemma the author assumes that $b(x)$ is continuous, but the assertion of the lemma also holds true in the case when $b(x)$ is locally bounded (condition \textbf{(C1)}). 
In particular, this automatically implies that $\process{X}$ is a strong Markov process with respect to the right-continuous and completed version of the underlying natural filtration.
Further, in \cite[Theorem V.21.1]{Rogers-Williams-Book-II-2000} it is shown that $$f(X_t)-f(X_0)-\int_0^t\mathcal{L}f(X_s)\,\D s\,,\qquad t\geq0\,,$$ is a $\mathbb{P}^x$-local martingale for every $x\in\R^d$ and every  $f\in C^2(\R^d)$, where $$\mathcal{L}f(x)\,:=\,\langle b(x),\nabla f(x)\rangle+\frac{1}{2}{\rm Tr}\,\sigma(x)\sigma(x)^T\nabla^2f(x)\,.$$  
If $b(x)$ and $\sigma(x)$ are continuous, then the infinitesimal generator $(\mathcal{A},\mathcal{D}_\mathcal{A})$ of $\process{X}$ (with respect to the Banach space $(B_b(\R^d),\lVert\cdot\rVert_\infty)$) satisfies $C_c^2(\R^d)\subseteq\mathcal{D}_\mathcal{A}$ and $\mathcal{A}|_{\mathcal{D}_\mathcal{A}}=\mathcal{L}.$ Here, $\lVert\cdot\rVert_\infty$ and  $C_c^2(\R^d)$ denote the supremum norm and the space of    twice continuously differentiable functions with compact support, respectively.
Recall, the infinitesimal generator (with respect to  $(\lVert\cdot\rVert_\infty,B_b(\R^d))$) of an $\R^d$-valued Markov process $\process{M}$ with semigroup  $\process{P}$ (defined as above) is a linear operator $\mathcal{A}:\mathcal{D}_\mathcal{A}\to B_b(\R^d)$ defined by
$$\mathcal{A}f\,:=\,
\lim_{t\to0}\frac{P_tf-f}{t}\,,\qquad f\in\mathcal{D}_{\mathcal{A}}\,:=\,\left\{f\in B_b(\R^d):
\lim_{t\to0}\frac{P_t f-f}{t} \ \textrm{exists in}\
\lVert\cdot\rVert_\infty\right\}\,.
$$
If $b(x)$ and $\sigma(x)$ are  Lipschitz continuous then $\process{X}$ is a $C_\infty$-Feller process, that is,  $P_t(C_\infty(\R^d))\subseteq C_\infty(\R^d)$ for all $t\geq0$ (see \cite[page 164]{Rogers-Williams-Book-II-2000}),
where  $C_\infty(\R^d)$   stands for the space of   continuous functions vanishing at infinity.

\subsection{Notation and preliminaries} \label{ss1}

We first recall some definitions and general results from the ergodic theory of Markov processes. Our main references are \cite{Meyn-Tweedie-AdvAP-II-1993} and \cite{Tweedie-1994}.
Let $(\Omega,\mathcal{F}, \process{\mathcal{F}},\process{\theta},$ $\process{M},\{\Prob^x\}_{x\in\R^d})$, denoted by  $\process{M}$ in the sequel, be a Markov process with c\`adl\`ag sample paths and  state space
$(\R^d,\mathcal{B}(\R^d))$ (see \cite{Blumenthal-Getoor-Book-1968}).  We let $p(t,x,\D y):=\Prob^x(M_t\in \D y)$, $t\ge0$, $x\in\R^d$,  denote the corresponding transition kernel.  For $t\ge0$ and  a  (not necessarily finite) measure $\mu$ on $\mathcal{B}(\R^d)$, $\mu P_t$ stands for $\int_{\R^d}p(t,x,\D y)\mu(\D x)$.
Also,  assume that $p(t,x,\D y)$
is a probability measure, that is, $\process{M}$ does not admit a cemetery point
in the sense of \cite{Blumenthal-Getoor-Book-1968}. Observe that this is not a restriction since,
as we have already commented, $\process{X}$ is non-explosive.
The process $\process{M}$ is called
\begin{enumerate}
	\item [(i)]
	$\phi$-irreducible if there exists a $\sigma$-finite measure $\phi$ on
	$\mathcal{B}(\R^d)$ such that whenever $\phi(B)>0$ we have
	$\int_0^{\infty}p(t,x,B)\D t>0$ for all $x\in\R^d$;
	
	\item [(ii)]
	transient if it is $\phi$-irreducible, and if there exists a countable
	covering of $\R^d$ with sets
	$\{B_j\}_{j\in\N}\subseteq\mathcal{B}(\R^d)$, and for each
	$j\in\N$ there exists a finite constant $\gamma_j\ge0$ such that
	$\int_0^{\infty}p(t,x,B_j)\,\D{t}\le \gamma_j$ holds for all $x\in\R^d$;
	
	\item [(iii)]
	recurrent if it is $\phi$-irreducible, and $\phi(B)>0$ implies
	$\int_{0}^{\infty}p(t,x,B)\,\D{t}=\infty$ for all $x\in\R^d$.
\end{enumerate}
Let us remark that if $\{M_t\}_{t\ge0}$ is a $\phi$-irreducible
Markov process, then the irreducibility measure $\phi$ can be
maximized. This means that there exists a unique ``maximal" irreducibility
measure $\psi$ such that for any measure $\Bar{\phi}$,
$\{M_t\}_{t\ge0}$ is $\Bar{\phi}$-irreducible if and only if
$\Bar{\phi}$ is absolutely continuous with respect to $\psi$ (see \cite[Theorem~2.1]{Tweedie-1994}).
In view to this, when we refer to an irreducibility
measure we actually refer to the maximal irreducibility measure.
It is also well known that every $\psi$-irreducible Markov
process is either transient or recurrent (see \cite[Theorem
2.3]{Tweedie-1994}). 
Further, recall that a Markov process $\process{M}$ is called 
\begin{enumerate}
	\item[(i)] open-set irreducible
	if the support of its maximal irreducibility measure $\psi$,  
	$${\rm supp}\,\psi\,=\,\{x\in\R^d: \psi(O)>0\ \text{for every open neighborhood}\ O\ \text{of}\ x\}\,,$$ has a non-empty interior;
	\item [(ii)] aperiodic if it admits an irreducible skeleton chain, that is, 
	there exist $t_0>0$ and a $\sigma$-finite measure $\phi$ on
	$\mathcal{B}(\R^d)$, such that $\phi(B)>0$ implies
	$\sum_{n=0}^{\infty} p(nt_0,x,B) >0$ for all $x\in\R^d$.
\end{enumerate}
A
(not necessarily finite) measure $\pi$ on $\mathcal{B}(\R^d)$ is called invariant for
$\process{M}$ if
$\pi P_t
= \pi$
for all $t\ge0$. 
It is well known that if $\process{M}$ is
recurrent, then it possesses a unique (up to constant
multiples) invariant measure $\pi$
(see \cite[Theorem~2.6]{Tweedie-1994}).
If the 
invariant measure is
finite, then it may be normalized to a probability measure. If
$\process{M}$ is recurrent with finite invariant measure, then $\process{M}$ is called 
positive recurrent; otherwise it is called null recurrent. Note that a transient 
Markov process cannot have a finite invariant measure. 
Indeed, assume that
$\process{M}$ is transient and that it admits a
finite invariant measure $\pi$, and fix some $t>0$.
Then, for each $j\in\N$, with $\gamma_j$ and $B_j$ as  above, we have
\begin{equation*}
t\pi(B_j) \,=\,
\int_0^{t}\pi P_s(B_j)\D{s}
\,\le\, \gamma_j\pi(\R^d)\,.
\end{equation*}
Now, by
letting $t\to\infty$ we obtain $\pi(B_j)=0$ for all
$j\in\N$, which is impossible.
A Markov process $\process{M}$ is called ergodic
if it possesses an invariant probability 
measure $\pi$ and there exists a nondecreasing function
$r:[0,\infty)\to[1,\infty)$ such that 
\begin{equation*}
\lim_{t\to\infty}r(t)\lVert p(t,x,\D {y})
-\pi(\D {y})\rVert_{{\rm TV}} \,=\,0\,,\qquad x \in \R^d\,,
\end{equation*} where $\lVert\mu\rVert_{{\rm TV}}:=\sup_{B\in\mathcal{B}(\R^d)}|\mu(B)|$ is the total variation norm  of a signed measure $\mu$ (on $\mathcal{B}(\R^d)$).
We say that $\process{M}$ is sub-geometrically ergodic if it is ergodic and 
$\lim_{t\to\infty}\ln r(t)/t=0$, and 
that it is geometrically ergodic if it is ergodic and
$r(t)=\E^{\kappa t}$ for some $\kappa>0$. 
Let us remark that (under the assumptions of $C_b$-Feller property, open-set irreducibility
and aperiodicity) ergodicity is equivalent 
to positive recurrence (see 
\cite[Theorem 6.1]{Meyn-Tweedie-AdvAP-II-1993},
and \cite[Theorems 4.1, 4.2 and 7.1]{Tweedie-1994}).

We now recall the notion and some general facts about Wasserstein distances  (on $\R^d$). Let $\rho$ be a metric on $\R^d$. Denote by $\R^d_\rho$ the topology induced by $\rho$, and let $\mathcal{B}(\R^d_\rho)$ be the corresponding Borel $\sigma$-algebra. 
For $p\ge0$ denote by $\mathcal{P}_{\rho,p}$ the space of all probability measures $\mu$ on $\mathcal{B}(\R_\rho^d)$ having finite $p$-th moment, that is, $\int_{\R^d}\rho(x_0,x)^p\mu(\D x)<\infty$ for some (and then  any) $x_0\in\R^d$.
Also,  $\mathcal{P}_{\rho,0}$  is denoted by $\mathcal{P}_\rho$. If  $\rho$ is the standard $d$-dimensional Euclidean  metric, then   $\mathcal{P}_{\rho,p}$  and $\mathcal{P}_{\rho}$ are denoted by   $\mathcal{P}_p$ and $\mathcal{P}$, respectively. For $p\ge1$ and  $\mu,\nu\in\mathcal{P}$, the $\mathcal{L}^p$-Wasserstein distance between $\mu$ and $\nu$ is defined as $$\W_{\rho,p}(\mu,\nu)\,:=\,\inf_{\Pi\in\mathcal{C}(\mu,\nu)}\left(\int_{\R^{d}\times\R^{d}}\rho(x,y)^p\,\Pi(\D x,\D y)\right)^{1/p}\,,$$ where  $\mathcal{C}(\mu,\nu)$ is the family of couplings of $\mu$ and $\nu$, that is, $\Pi\in\mathcal{C}(\mu,\nu)$ if and only if $\Pi$ is a probability measure on $\R^{d}\times\R^{d}$ having $\mu$ and $\nu$ as its marginals. It is not hard to see that $\W_{\rho,p}$ satisfies the axioms of a (not necessarily finite) distance on $\mathcal{P}_{\rho}$. The restriction  of $\W_{\rho,p}$ to  $\mathcal{P}_{\rho,p}$   defines a finite distance.
If $(\R^d,\rho)$ is a Polish space, then it is well known that $(\mathcal{P}_{\rho,p},\W_{\rho,p})$ is also a Polish space (see \cite[Theorem 6.18]{Villani-Book-2009}). Of our special interest will be the situation when $\rho$ takes the form $\rho(x,y)=f(|x-y|)$,  where
$f:[0,\infty)\to[0,\infty)$ is a non-decreasing concave function satisfying 
$f(t)=0$ if and only if $t=0$. In this situation, the corresponding Wasserstein space is denoted by $(\mathcal{P}_{f,p},\W_{f,p})$ (which does not have to be a Polish space).
Observe that if $f(t)=\Ind_{(0,\infty)}(t)$, then $\W_{f,p}(\mu,\nu)=\rVert\mu-\nu\lVert_{{\rm TV}}$ for all $p\ge1$. In the case when $f(t)=t$, the corresponding Wasserstein space is denoted just by $(\mathcal{P}_{p},\W_{p})$ (which is always a Polish space). 
For more on Wasserstein distances we refer the readers to \cite{Villani-Book-2009}.

\subsection{Main results}
The main goal of this article is to obtain (sharp) conditions for sub-geometric ergodicity of $\process{X}$ with respect to the toal variation distance and/or a class of Wasserstein distances. Before stating the main results, we introduce some notation we  need in the sequel. 
Fix $x_0\in\R^d$ and $r_0\ge0$, and put
\begin{align*}
c(x)&\,:=\,\sigma(x)\sigma(x)^{T}\,,\\
A(x)&\,:=\,\frac{1}{2}\,{\rm Tr}\,c(x)\,,\qquad x\in\R^{d}\,,\\
B_{x_0}(x)&\,:=\,\langle x-x_0,b(x)\rangle\,,\qquad x\in\R^{d}\,,\\
C_{x_0}(x)&\,:=\,\frac{\langle x-x_0,c(x)(x-x_0)\rangle}{|x-x_0|^{2}}\,,\qquad x\in\R^{d}\setminus\{x_0\}\,,\\
\gamma_{x_0}(r)&\,:=\,\inf_{|x-x_0|=r}C_{x_0}(x)\,,\qquad r>0\,,\\
\iota_{x_0}(r)&\,:=\,\sup_{|x-x_0|=r}\frac{2A(x)-C_{x_0}(x)+2B_{x_0}(x)}{C_{x_0}(x)}\,,\qquad r>0\,,\\
I_{x_0}(r)&\,:=\,\int_{r_0}^{r}\frac{\iota_{x_0}(s)}{s}ds\,,\qquad r\geq r_0\,.
\end{align*}

\begin{theorem}\label{tm:TV} Assume $\textbf{(C1)-(C3)}$, and assume that $\process{X}$ is  open-set irreducible and aperiodic. 
	Further, let $\varphi:[1,\infty)\to(0,\infty)$ be a  non-decreasing, differentiable and concave  function satisfying $\lim_{t\to \infty}\varphi'(t)=0$ and
	\begin{equation}\label{eq2}\Lambda\,:=\,\int_{r_0}^{\infty}\varphi\left(\int_{r_0}^{u}\E^{-I_{x_0}(v)}\D v+1\right)\frac{\E^{I_{x_0}(u)}}{\gamma_{x_0}(u)}\D u\,<\,\infty\end{equation}
	for some $x_0\in\R^d$ and $r_0\geq0$, 
	and assume that $c(x)$ is positive definite for all $x\in\R^d$, $|x-x_0|\geq r_0$ (hence, the above functions and the relation in \eqref{eq2} are well defined).	
	Then, $\process{X}$ admits a unique invariant  $\pi\in\mathcal{P}$ satisfying 
	$$\lim_{t\to\infty}\varphi(\varPhi^{-1}(t))\lVert\delta_x P_t-\pi\rVert_{{\rm TV}}\,=\,0\,,\qquad x\in\R^d\,,$$
	where  $$\varPhi(t)\,:=\,\int_1^{t}\frac{\D s}{\varphi(s)}\,,\qquad t\geq1\,.$$  
\end{theorem}

The   proof of Theorem \ref{tm:TV} is based on the Foster-Lyapunov method for sub-geometric ergodicity of Markov processes developed in \cite{Douc-Fort-Guilin-2009}.
The method itself consists of finding an appropriate recurrent set $\mathcal{C}\in\mathcal{B}(\R^d)$, and constructing an appropriate  function $\V:\R^d\to[1,\infty)$ (the so-called Lyapunov (energy) function) contained in the domain of the extended generator $\mathcal{A}^{e}$ of the underlying Markov process $\process{M}$ (see \cite[Section 1]{Meyn-Tweedie-AdvAP-III-1993} for details), such that the Lyapunov equation \begin{equation}
\label{eq:lyap}
\mathcal{A}^{e}\V(x)\,\le\,-\varphi(\V(x))+\beta\Ind_\mathcal{C}(x)\,,\qquad x\in\R^d\,,\end{equation} holds for some $\beta\in\R$ (see \cite[Theorem 3.4]{Douc-Fort-Guilin-2009}). 
The equation in \eqref{eq:lyap} implies that 
for any $\delta>0$ the $\varphi\circ\Phi^{-1}$-moment  of the $\delta$-shifted hitting time $\tau_\mathcal{C}^\delta:=\inf\{t\ge\delta:M_t\in \mathcal{C}\}$ of $\process{M}$ on  $\mathcal{C}$ (with respect to $\mathbb{P}^x$) is finite and controlled by $\V(x)$ (see \cite[Theorem 4.1]{Douc-Fort-Guilin-2009}). 
However, this property in general does not immediately imply ergodicity of $\process{M}$. Namely, we also need to ensure that a similar property holds for any  other ``reasonable'' set. 
If $\process{M}$ is $\psi$-irreducible and  $\mathcal{C}$ is a petite
set, then indeed for any $\delta>0$  the
$\varphi\circ\Phi^{-1}$-moment of  $\tau_B^\delta,$ for any $B\in\mathcal{B}(\R^d)$ with $\psi(B)>0$,    is again   finite and controlled by $\V(x)$
(see \cite[the discussion after Theorem 4.1]{Douc-Fort-Guilin-2009}).
Recall, a set $C\in\mathcal{B}(\R^d)$ is said to be   petite if it satisfies a Harris-type minorization condition: there are a probability   measure $\eta_C$ on $\mathcal{B}((0,\infty))$ (the standard Borel $\sigma$-algebra on $(0,\infty)$) and a non-trivial measure $\nu_C$ on $\mathcal{B}(\R^d)$, such that $\int_0^\infty p(t,x,B)\eta_C(\D t)\ge\nu_C(B)$ for all $x\in C$ and $B\in\mathcal{B}(\R^d)$. Recall also that  $\psi$-irreducibility implies that the state space (in this case $(\R^d,\mathcal{B}(\R^d)$) can be covered by a countable union of petite sets (see \cite[Propositio 4.1]{Meyn-Tweedie-AdvAP-II-1993}.    Also, $C_b$-Feller property and open-set irreducibility of $\process{M}$ ensure that every compact set is  petite  (see \cite[Theorem Theorems 5.1 and 7.1]{Tweedie-1994}.
Intuitively,  petite sets take a role of singletons for Markov processes on non-discrete state spaces (see \cite[Section 4]{Meyn-Tweedie-AdvAP-II-1993} and \cite[Chapter 5]{Meyn-Tweedie-Book-2009} for details).
However, as in the discrete setting, $\process{M}$ can also show certain cyclic behavior which  causes ergodicity not to hold (see \cite[Section 5]{Meyn-Tweedie-AdvAP-II-1993} and \cite[Chapter 5]{Meyn-Tweedie-Book-2009}). By assuming aperiodicity (which excludes this type of behavior), the sub-geometric ergodicity of $\process{M}$ follows from   \cite[Theorem 1]{Fort-Roberts-2005}, which states that finiteness of the $\varphi\circ\Phi^{-1}$-moment of $\tau_\mathcal{C}^\delta$ implies sub-geometric ergodicity of $\process{M}$ with rate $r(t)=\varphi(\Phi^{-1}(t))$. 
Let us remark that, in the context of the process $\process{X}$, the relation in \eqref{eq2} is crucial in   the construction  of (actually it appears as a part of) the appropriate Lyapunov function (see the proof of Theorem \ref{tm:TV}). Thus, through this relation we control the  
$\varphi\circ\Phi^{-1}$-moment of $\tau_\mathcal{C}^\delta$ with $\mathcal{C}$ being a closed ball around the origin with large enough radius.
We also remark that using  an analogous approach as above in \cite[Chapter 4]{Khasminskii-Book-2012}  positive recurrence of the process $\process{X}$ with globally Lipschitz  coefficients and with $c(x)$ being positive definite (hence, according to Theorem \ref{tm:IRAP}, $\process{X}$ is open-set irreducible and aperiodic) has been discussed.  Based on this result, and  analyzing polynomial moments of hitting times of compact sets,
in \cite[Theorem 6]{Veretennikov-1997}  polynomial ergodicity of $\process{X}$  has been obtained. In the follow up work, by using analogous techniques the same author established polynomial ergodicity of $\process{X}$ without directly assuming $\psi$-irreducibility and aperiodicity of the process, but basing on a local irreducibility condition which we discuss below (see \cite[Theorem 6]{Veretennikov-2000}).

An alternative and, in a certain sense, more general approach to this problem is  based on a local irreducibility condition. In this approach, instead of \eqref{eq:lyap}, we assume a slightly more general form of the Lyapunov equation:
\begin{equation}\label{eq:lyap2}
\mathcal{A}^{e}\V(x)\,\le\,-\varphi(\V(x))+\beta\,,\qquad x\in\R^d\,,\end{equation} for some $\beta\in\R$, and instead of 
assuming $\psi$-irreducibility and aperiodicity of $\process{M}$, we assume the so-called (local) Dobrushin condition (also known as Markov-Dobrushin condition): the Lyapunov function $\V(x)$ has precompact sub-level sets, and for every  $\gamma>0$ there is $t_\gamma>0$ such that \begin{equation}\label{eq:dobr}\sup_{(x,y)\in \{(u,v):\,\V(u)+\V(v)\le\gamma\}} \lVert p(t_\gamma,x,\D z)-p(t_\gamma,y,\D z) \rVert_{{\rm TV}}<1\,,\end{equation} see \cite[Theorem 4.1]{Hairer-Lecture-notes-2016} (see also \cite[Chapter 1.4]{Kulik-Book-2015} and \cite[Chapter 3]{Kulik-Book-2018}).
Observe that this condition actually means that for each $(x,y)\in \{(u,v):\V(u)+\V(v)\le\gamma\}$ the probability measures  $p(t_\gamma,x,\D z)$ and $p(t_\gamma,y,\D z)$ are not mutually singular. Intuitively, the Dobrushin condition encodes $\psi$-irreducibility and aperiodicity of $\process{M}$, and petiteness of sub-level sets of $\V(x)$. By using a coupling approach with an appropriately chosen Markov coupling of $\process{M}$, say  $\process{M^c}$, the Lyapunov equation and Dobrushin condition, analogously as before, imply that  the hitting (that is, coupling) time
$\tau_{c}:=\inf\{t\ge0:M^c_t\in {\rm diag}\}$ of $\process{M^c}$ on  ${\rm diag}:=\{(x,x):x\in\R^d\}$  is a.s. finite (with respect to the probability measure  corresponding to $\process{M^c}$ with any initial position $(x,y)\in\R^d\times \R^d$). Moreover, it follows that
the $\Phi^{-1}$-moment  of  $\tau_{c}$  is finite and controlled by $\V(x)+\V(y)$. Then from the coupling inequality it follows that $\process{M}$ admits a unique invariant $\pi\in\mathcal{P},$ and $$\sup_{t\ge 0}\varphi(\Phi^{-1}(t))\lVert p(t,x,\D y) -\pi(\D y) \rVert_{{\rm TV}}<\infty\,,\qquad x\in\R^d\,,$$ (see \cite[Theorem 4.1]{Hairer-Lecture-notes-2016}, or \cite[Chapter 1.4]{Kulik-Book-2015} and \cite[Chapter 3]{Kulik-Book-2018} for the skeleton chain approach). 

Observe that \eqref{eq:lyap2} follows from \eqref{eq:lyap}. Also,  $\psi$-irreducibility and aperiodicity (together with \eqref{eq:lyap}) imply that the Dobrushin condition holds on the Cartesian product of any petite set with itself. Namely, according to \cite[Proposition 6.1]{Meyn-Tweedie-AdvAP-II-1993}, for any petite set $C$ there is $t_C>0$ such that for the measure $\eta_C$ (in the definition of petiteness)   the Dirac measure in $t_C$ can be taken (with some, possibly different, non-trivial measure $\nu_C$). Thus, $p(t_C,x,B)\ge\nu_C(B)$ for any $x\in C$ and $B\in\mathcal{B}(\R^d)$, which implies  \begin{equation}\label{eq:dob2}\sup_{(x,y)\in C\times C} \lVert p(t_C,x,\D z)-p(t_C,y,\D z) \rVert_{{\rm TV}}<1\,.\end{equation} If in addition $\process{M}$ is
$C_b$-Feller  and open-set irreducible, as we have already commented, every compact set is petite so the above relation holds for any bounded set $C$, showing that, at least in this particular situation, the approach based on the Dobrushin condition is more general than the approach based on $\psi$-irreducibility and aperiodicity.
Situations where it shows a clear advantage are discussed in \cite{Kulik-2009} and \cite{Abourashchi-Veretennikov-2010}. In the first reference the author considers
a Markov process obtained as a solution to a L\'evy-driven SDE with highly irregular coefficients and noise term, while in the second a diffusion process with highly irregular (discontinuous) drift function and uniformly elliptic diffusion coefficient has been considered. In these concrete situations it is not clear whether one can obtain $\psi$-irreducibility and aperiodicity of the processes, whereas the authors obtain \eqref{eq:dob2} for any compact set $C$ (see \cite[Theorem 1.3]{Kulik-2009} and \cite[Lemma 3]{Abourashchi-Veretennikov-2010}).  For more on ergodic properties of Markov processes based on the Dobrushin condition we refer the readers to \cite{Hairer-Lecture-notes-2016}, \cite{Kulik-Book-2015} and \cite{Kulik-Book-2018}.

In the case of the process $\process{X}$, open-set irreducibility and aperiodicity   will be   satisfied if the coefficient $c(x)$ is Lipschitz continuous and uniformly elliptic (see the discussion after Proposition \ref{p:4}). In Theorem \ref{tm:IRAP} we  show that $\process{X}$ will be open-set irreducible and aperiodic if $b(x)$ and $c(x)$  are H\"older continuous, and $c(x)$ is uniformly elliptic on an open ball only.
Let us also remark that, without further regularity assumptions on $b(x)$ and $c(x)$, it is not clear how to check the Dobrushin condition in these two situations.

The problem of sub-geometric ergodicity of diffusion processes (with respect to the total variation distance) has  already been considered in the literature (see   \cite{Douc-Fort-Guilin-2009}, \cite{Fort-Roberts-2005}, \cite{Kulik-Book-2015},
\cite{Kulik-Book-2018}, \cite{Sandric-ESAIM-2016}, \cite{Veretennikov-1997} and \cite{Veretennikov-2000}. 
In these works it has been shown that $\process{X}$ will be sub-\linebreak geometrically
ergodic with rate $t^{\alpha/(1-\alpha)}$ (that is, $\varphi(t)=t^\alpha$),  $\alpha\in(0,1)$, if there exist $\gamma>0$, $\Gamma>0$ and $r_0\ge0$, such that \begin{equation}\label{eq:subgeo}
A(x)-\left( 1-\frac{\gamma}{2}\right)C_0(x)+B_0(x)\,\le\,-\Gamma|x|^{\gamma\alpha-\gamma+2}\,,\qquad |x|\geq r_0\,.
\end{equation} 
However, this  result is far for being sharp (optimal). Namely, in Proposition \ref{p1} we show that \eqref{eq:subgeo}  implies \eqref{eq2}, and in Example \ref{e:TV} we give an example of a diffusion process satisfying conditions from Theorem \ref{tm:TV}, but not the condition in \eqref{eq:subgeo}.

On the other hand, in the case when $c(x)$ is not regular enough, the topology induced by the total variation distance becomes too ``rough", that is, it cannot completely capture the singular behavior of $\process{X}$. In oder words, $p(t,x,\D y)$ cannot converge to the underlying invariant probability measure  (if it exists) in this topology, but in a weaker sense  (see \cite{Sandric-RIM-2017} and the references therein). 
Therefore, in this situation, we naturally resort to  Wasserstein distances which, in a certain sense, induce a finer topology, that is, convergence with respect to a Wasserstein distance implies the weak convergence of probability measures (see \cite[Theorems 6.9 and 6.15]{Villani-Book-2009}).

\begin{theorem}\label{tm:WASS1}
	Let $\sigma(x)\equiv\sigma$ be an arbitrary $d\times n$ matrix, and assume $\textbf{(C1)-(C3)}$.  Further, let  $p\ge1$ and let $f,\psi:[0,\infty)\to[0,\infty)$ 
	be such that
	\begin{itemize}
		\item [(i)] $f(t)$ is concave, non-decreasing, absolutely continuous on $[t_0,t_1]$ for any $0<t_0<t_1<\infty$, and $f(t)=0$ if and only if $t=0$;
		\item[(ii)] $\psi(t)$ is convex and $\psi(t)=0$ if and only if $t=0$;
		\item[(iii)] there are $\gamma>0$, $\Gamma>0$ and $t_0>0$, such that
		$f(t_0)\le \gamma$ and
		\begin{equation}\label{eqWASS1}f'(|x-y|)\langle x-y,b(x)-b(y)\rangle\,\le\,\left\{\begin{array}{cc}
		-\Gamma|x-y|\,\psi(f(|x-y|))\,, & f(|x-y|)\le \gamma\,, \\
		0\,,& f(|x-y|)> \gamma\,,
		\end{array}\right.\end{equation} a.e. on $\R^d$.
	\end{itemize}  
	Then, 
	\begin{itemize}
		\item [(a)] for all $x,y\in\R^d$, $f(|x-y|)\le \gamma$, it holds that \begin{equation}\label{eqWASS2}\W_{f,p}(\delta_x P_t,\delta_y P_t)\,\le\, \Psi_{f(|x-y|)}^{-1}(\Gamma t)\,,\qquad t\geq 0\,,\end{equation} where $\Psi_\kappa(t):=\int_t^\kappa\frac{\D s}{\psi(s)}$ for $\kappa>0$ and $t\in(0,\kappa].$
		\item [(b)] for all $x,y\in\R^d$, $f(|x-y|)\le \gamma$, and all $\kappa\ge \gamma$ it holds that \begin{equation}\label{eqWASS3}\W_{f,p}(\delta_x P_t,\delta_y P_t)\,\le\, \Psi_{\kappa}^{-1}(\Gamma t)\,,\qquad t\geq 0\,.\end{equation} In addition, if $\Psi_\infty(t):=\int_t^\infty\frac{\D s}{\psi(s)}<\infty$ for  $t\in(0,\infty),$ then  \begin{equation}\label{eqWASS4}\W_{f,p}(\delta_x P_t,\delta_y P_t)\,\le\, \Psi_{\infty}^{-1}(\Gamma t)\,,\qquad t\geq 0\,.\end{equation}
		\item [(c)] for any $x,y\in\R^d$ it holds that \begin{equation}\label{eqWASS5}\W_{f,p}(\delta_x P_t,\delta_y P_t)\,\le\, \lceil\delta|x-y|\rceil\Psi_{\gamma}^{-1}(\Gamma t)\,,\qquad t\geq 0\,,\end{equation} where  $\delta:=\inf\{t>0:f(t^{-1})\le \gamma\}$ and  $\lceil u\rceil$ denotes the least integer greater than or equal to $u\in\R.$ Also, according to (b), $\Psi_{\gamma}^{-1}(\Gamma t)$ in \eqref{eqWASS5} can be replaced by $\Psi_{\kappa}^{-1}(\Gamma t)$ for any $\kappa\ge \gamma$, and by $\Psi_{\infty}^{-1}(\Gamma t)$ if $\Psi_{\infty}(t)<\infty$ for $t\in(0,\infty)$.  
	\end{itemize}  
\end{theorem}

Observe that $f(t)$ is $\mathcal{B}((0,\infty))$-measurable, implying that the relation in \eqref{eqWASS2} is well defined. 
The proof of Theorem  \ref{tm:WASS1} is based on the  so-called synchronous coupling method (see \cite[Example 2.16]{Chen-Book-2005} for details) and the
asymptotic flatness   condition given in  \eqref{eqWASS1}. Let us remark that in a special case when $p=2$ and $f(t)=\psi(t)=t$ in \cite{Renesse-Sturm-2005} it has been shown that the relation in  \eqref{eqWASS2}  (observe that in this case  $\Psi^{-1}_{f(|x-y|)}(\Gamma t)=|x-y|\E^{-\Gamma t}$) is equivalent to the    asymptotic flatness  condition (in the sense of \cite{Arapostathis-Borkar-Ghosh_Book-2012}) 
\begin{equation}\label{eqWASS6}\langle x-y,b(x)-b(y)\rangle\,\le\,
-\Gamma|x-y|^2\,,\qquad x,\,y\in\R^d\,.\end{equation} Even though at first sight  the condition in  \eqref{eqWASS1} seems to be less restrictive than the condition in \eqref{eqWASS6}, they are actually equivalent. This can be easily observed by taking an equidistant subdivision of the line segment connecting $x$ and $y$, such that the distance between consecutive points is strictly less than $\gamma$, and then applying triangle inequality. 
On the other hand, in the case when $\psi(t)$ is not  the identity function this does not hold in general. Namely, $\psi(t)$ is not sub-additive, but super-additive. 
A typical example of a drift function (in dimension $d=1$)  satisfying \eqref{eqWASS1} (and \eqref{eqWASS7}), but not \eqref{eqWASS6}, is $b(x)=-{\rm sgn}(x)|x|^p$, $p> 1$, together with $f(t)=t$ and $\psi(t)=|t|^p$ (see Example \ref{ex1}). More generally,  no drift function that is sub-linear near the origin can satisfy \eqref{eqWASS6}, but it might satisfy 
\eqref{eqWASS1}.

Finally, as a consequence of Theorem \ref{tm:WASS1} we conclude the following.

\begin{theorem}\label{tm:WASS2}
	In addition to the assumptions of Theorem \ref{tm:WASS1} with $f(t)=t$, assume
	\begin{equation}\label{eqWASS7}\langle x-y,b(x)-b(y)\rangle\,\le\,
	-\Gamma|x-y|\,\psi(|x-y|)\,,\qquad x,\,y\in\R^d\,.\end{equation} 
	Then, the process $\process{X}$ admits a unique invariant  $\pi\in\cap_{p\ge1}\mathcal{P}_p$, and for any $\kappa>0$, $p\ge1$ and $\mu\in\mathcal{P}_p$, 
	\begin{equation}\label{eqWASS8}\W_p(\mu P_t,\pi)\,\le\, \left(\frac{\W_p(\mu,\pi)}{\kappa}+1\right)\Psi_\kappa^{-1}(\Gamma t)\,,\qquad t\geq 0\,.\end{equation} 
\end{theorem}

Let us also remark that if $\sigma(x)\equiv\sigma$ is quadratic and non-singular matrix, and $b(x)$ satisfies the following  asymptotic flatness condition 
\begin{equation}\label{eqWASS9}\langle x-y,b(x)-b(y)\rangle\leq\left\{\begin{array}{cc}
\Gamma_1|x-y|^2\,, & |x-y|\le \Delta\,, \\
-\Gamma_2|x-y|^2\,,& |x-y|\ge \Delta\,,
\end{array}\right.\qquad x,\,y\in\R^d\,,\end{equation} for some $\Gamma_1>0$, $\Gamma_2>0$ and $\Delta>0$, 
by using the so-called coupling by reflection method (see \cite[Example 2.16]{Chen-Book-2005} for details), in \cite{Eberle-2011} (see also \cite{Eberle-2015} and \cite{Luo-Wang-2016}) it has been shown that there is a concave function $f(t)$ (given explicitly in terms of the constants $\Gamma_1,$ $\Gamma_2$ and $\Delta$, and coefficients $\sigma$ and $b(x)$) defining a metric $\rho(x,y)=f(|x-y|)$ on $\R^d$ under which $\process{X}$ satisfies contraction property of the type \eqref{eqWASS5} with geometric rate of convergence, and geometric ergodicity property of the type \eqref{eqWASS8}. As we have already commented, $b(x)=-{\rm sgn}(x)|x|^p$, $p> 1$, satisfies \eqref{eqWASS1} and \eqref{eqWASS7}, but clearly it also  satisfies \eqref{eqWASS9}. However, in the later case, in order to conclude contractivity or ergodicity  it is necessary to assume non-singularity of $\sigma$, while in the former case we can allow $\sigma$ to be singular. Let us also remark that in the case when $\sigma$ is non-singular, by taking $y=0$ in \eqref{eqWASS9}, one can easily see that $\process{X}$ is geometrically ergodic with respect to the total variation distance (see Proposition \ref{p:4}).

\subsection{Literature review}

Our work relates to the active research on ergodicity properties of Markov processes, and the vast literature on SDEs. 
In \cite{Arapostathis-Borkar-Ghosh_Book-2012}, \cite{Bhattacharya-1978}, \cite{Kulik-Book-2015}, \cite{Kulik-Book-2018}, \cite{Stramer-Tweedie-1997} and \cite{Veretennikov-1997}   ergodicity properties with respect to the total variation distance of diffusion processes are established using the Foster-Lyapunov(-type) method. In this article, we  generalize the ideas from \cite{Bhattacharya-1978} (see also \cite[Chapter 9]{Friedman-Book-1975} and \cite[Supplement]{Khasminskii-1960}) and obtain sharp conditions
which ensure  ergodicity properties with sub-geometric rates of convergence of  this class of processes. Furthermore, we adapt these results and discuss also ergodicity properties of a class of diffusion processes with jumps and a class of Markov processes obtained through the Bochner's subordination. These results are related to 
\cite{Albeverio-Brzezniak-Wu-2010}, \cite{Arapostathis-Pang-Sandric-2019}, \cite{Deng-Schilling-Song-2017, Deng-Schilling-Song-Erratum-2018} \cite{Douc-Fort-Guilin-2009}, \cite{Down-Meyn-Tweedie-1995},
\cite{Fort-Roberts-2005}, \cite{Kevei-2018},  \cite{Kulik-2009},   \cite{Masuda-2007, Masuda-Erratum-2009}, \cite{Meyn-Tweedie-AdvAP-II-1993}, \cite{Meyn-Tweedie-AdvAP-III-1993}, \cite{Sandric-ESAIM-2016}, \cite{Wang-2008}, \cite{FYWang-2011}, \cite{Wang-2011}   and \cite{Wee-1999}  
where the ergodicity properties  of  general 
Markov processes are established using the  Foster-Lyapunov method again.

The studies on ergodicity properties with respect to the total variation distance assume that 
the Markov processes are irreducible and aperiodic. This is satisfied if the process does not show a singular behavior in its motion, that is, its diffusion part is non-singular and/or its jump part shows enough jump activity.  
For Markov processes that do not converge in total variation, 
ergodic properties under  Wasserstein distances are studied 
since they may converge weakly under certain conditions, see \cite{Bolley-Gentil-Guillin-2012}, 
\cite{Butkovsky-2014}, \cite{Eberle-2011}, \cite{Eberle-2015}, \cite{Hairer-Mattingly-Scheutzow-2011}, \cite{Luo-Wang-2016}, \cite{Majka-2017}, \cite{Renesse-Sturm-2005} and \cite{Wang-2016}. 
In \cite{Bolley-Gentil-Guillin-2012} and \cite{Renesse-Sturm-2005}, the coupling approach and the asymptotic flatness property in \eqref{eqWASS6} are employed to establish geometric
contractivity and  ergodicity of the semigroup 
of a diffusion process with possibly singular diffusion coefficient, with respect to a Wasserstein distance. 
However, in many situations the condition in \eqref{eqWASS6} is too restrictive. For example,  as we have already commented, drift functions which are sub-linear near the origin do not satisfy \eqref{eqWASS6}. 
The first step in relaxing this condition has been recently
done in \cite{Eberle-2011} (see also \cite{Eberle-2015} and \cite{Luo-Wang-2016}) where \eqref{eqWASS6} is replaced by the asymptotic flatness  property in \eqref{eqWASS9},  but at
the price of assuming that the diffusion coefficient is non-singular. Under these assumptions geometric
contractivity and  ergodicity of the semigroup of a diffusion process with respect to a Wasserstein distance 
are again established.
In this article, we  relax \eqref{eqWASS6} to the asymptotic flatness conditions in \eqref{eqWASS1} and \eqref{eqWASS7}, and obtain  sub-geometric
contractivity and sub-geometric ergodicity of the semigroup of a diffusion process, with possibly singular diffusion coefficient, with respect to a Wasserstein distance. At the end,  we again discuss  ergodicity properties, but with respect to Wasserstein distances, of a class of diffusion processes with jumps and a class of Markov processes obtained through the Bochner's subordination.

At the end we  remark that an analogous results, with respect to the total variation distance and Wasserstein distances, have also been obtained in the discrete-time setting, see   \cite{Douc-Moulines-Priouret-Soulier_Book-2018}, \cite{Durmus-Fort-Moulines-2016}, \cite{Durmus-Fort-Moulines-Soulier-2004}, \cite{Fort-Moulines-2003}, \cite{Kulik-Book-2015}, \cite{Kulik-Book-2018}, \cite{Meyn-Tweedie-Book-2009}, \cite{Veretennikov-1997}, \cite{Veretennikov-2000}, \cite{Tuominen-Tweedie-1994} and the references therein. 

\subsection{Organization of the article}
In the next section, we prove Theorem \ref{tm:TV}, and discuss open-set irreducibility and aperiodicity of diffusion processes. Also, we discuss sub-geometric ergodicity  of two classes of Markov processes with jumps. 
In Section \ref{s3}, we prove Theorems \ref{tm:WASS1} and  \ref{tm:WASS2}, and again discuss sub-geometric ergodicity of  Markov processes with jumps, but with respect to Wasserstein distances.

\section{Ergodicity with respect to the total variation distance} \label{s2}

In this section, we first prove Theorem \ref{tm:TV}.
Then, we discuss open-set irreducibility and aperiodicity of diffusion processes. Finally, at the end, we discuss  sub-geometric ergodicity  of two classes of Markov processes with jumps.  

\subsection{Ergodicity of diffusion processes}

We  start with the proof of Theorem \ref{tm:TV}.

\begin{proof}[Proof of Theorem \ref{tm:TV}] Set $ \varphi_\Lambda(t)=\varphi(t)/\Lambda,$ where $\Lambda$ is given in \eqref{eq2},  and observe that $\varphi_\Lambda(t)$ has the same properties as $\varphi(t)$. Next, define
	$$\bar{\V}(r)\,:=\,\int_{r_0}^{r}\E^{-I_{x_0}(u)}\int_u^{\infty}\varphi_\Lambda\left(\int_{r_0}^{v}\E^{-I_{x_0}(w)}\D w+1\right)\frac{\E^{I_{x_0}(v)}}{\gamma_{x_0}(v)}\D v\,\D u\,,\qquad r\geq r_0\,.$$  Clearly, for  $r\ge r_0$ it holds that
	\begin{equation}
	\label{E:en}\bar{\V}(r)\,\le\,\int_{r_0}^{r}\E^{-I_{x_0}(u)}\D u\,, \end{equation}
	and \begin{align*}\bar{\V}'(r)&\,=\,\E^{-I_{x_0}(r)}\int_r^{\infty}\varphi_\Lambda\left(\int_{r_0}^{u}\E^{-I_{x_0}(v)}\D v+1\right)\frac{\E^{I_{x_0}(u)}}{\gamma_{x_0}(u)}\D u\\
	\bar{\V}''(r)&\,=\,-\frac{\iota_{x_0}(r)}{r}\E^{-I_{x_0}(r)}\int_r^{\infty}\varphi_\Lambda\left(\int_{r_0}^{u}\E^{-I_{x_0}(v)}\D v+1\right)\frac{\E^{I_{x_0}(u)}}{\gamma_{x_0}(u)}\D u-\frac{\varphi_\Lambda\left(\int_{r_0}^{r}\E^{-I_{x_0}(u)}\D u+1\right)}{\gamma_{x_0}(r)}\,.
	\end{align*}
	Further, fix $r_1>r_0$  and let $\V:\R^{d}\to[0,\infty)$, $\V\in C^2(\R^d)$, be such that $\V(x)=\bar{\V}(|x-x_0|)+1$ for $x\in\R^d,$ $|x-x_0|\ge r_1.$ Now, for $x\in\R^{d}$, $|x-x_0|\ge r_1$, we have
	\begin{align*}
	\mathcal{L}\V(x)&\,=\,\frac{1}{2}C_{x_0}(x)\bar{\V}''(|x-x_0|)+\frac{\bar{\V}'(|x-x_0|)}{2|x-x_0|}(2A(x)-C_{x_0}(x)+2B_{x_0}(x))\\
	&\,\le\, -\frac{1}{2}\varphi_\Lambda\left(\int_{r_0}^{|x-x_0|}\E^{-I_{x_0}(u)}\D u+1\right)\\
	&\,\le\,-\frac{1}{2}\varphi_\Lambda(\V(x))\,,
	\end{align*}
	where in the final step we employed the fact that $\varphi(t)$ (that is, $\varphi_\Lambda(t)$) is non-decreasing and \eqref{E:en}. Thus, we have obtained the relation in (3.11) in \cite[Theorem 3.4 (i)]{Douc-Fort-Guilin-2009} with $\phi(t)=\varphi_\Lambda(t)$, $C=\bar B_{r_1}(x_0)$ (the topological closure of the open ball $B_{r_1}(x_0)$), and $b=\sup_{x\in C}|\mathcal{L}V(x)|$. Now, \cite[Theorems 5.1 and 7.1]{Tweedie-1994}, together with open-set irreducibility, aperiodicity and $C_b$-Feller property  of $\process{X}$,  imply that $\process{X}$ meets the conditions of \cite[Theorem 3.2]{Douc-Fort-Guilin-2009} with $\Psi_1(t)=t$ and $\Psi_2(t)=1$, which concludes the proof.  
\end{proof}

As a direct consequence of Theorem \ref{tm:TV} we conclude the following.
\begin{corollary}\label{c:TV}
	If in Theorem \ref{tm:TV} we take $\varphi(t)=t^{\alpha}$ with $\alpha\in(0,1)$, then  $\process{X}$  is sub-geometrically
	ergodic with rate $t^{\alpha/(1-\alpha)}.$
\end{corollary}

If $\varphi(t)$ is bounded then the condition in \eqref{eq2} reduces to $$\int_{r_0}^\infty\frac{\E^{I_{x_0}(u)}}{\gamma_{x_0}(u)}\D u\,<\,\infty\,,$$  which is exactly the  condition for ergodicity obtained in \cite[Theorem 3.5]{Bhattacharya-1978} (see also \cite[Theorem 1.2]{Wang-2008} and \cite[Chapter IV]{Mandl-Book-1968} for the one-dimensional case).
By taking $\varphi(t)=t$  one expects to obtain geometric ergodicity of $\process{X}$. However, we 
cannot  apply Theorem \ref{tm:TV} directly since $\lim_{t\to\infty}\varphi'(t)\neq0$. 
By employing analogous  ideas as in  Theorem \ref{tm:TV}, in \cite[Theorem 1.3]{Wang-2008} the author proves geometric ergodicity of $\process{X}$ under \eqref{eq2} (with $\varphi(t)=t$) in the one-dimensional case. In what follows we give a multi-dimensional version of this result.

\begin{proposition}\label{p:4} If in Theorem \ref{tm:TV}
	$\liminf_{t\to \infty}\varphi'(t)>0$,
	then $\process{X}$ is geometrically ergodic.	
\end{proposition}
\begin{proof} 
	First, observe that since $\varphi(t)$ is  differentiable and concave, $t\mapsto\varphi'(t)$ is non-increasing. Thus, since $\varphi(t)$ is also non-decreasing,
	there are constants $\Gamma\ge \gamma>0$ such that $$\gamma t-\gamma+\varphi(1)\,\le\,\varphi(t)\,\le\,\Gamma t-\Gamma+\varphi(1)\,,\qquad t\ge1\,.$$  Consequently, the condition in \eqref{eq2} is equivalent to 
	\begin{equation*}\int_{r_0}^{\infty}\left(\int_{r_0}^{u}\E^{-I_{x_0}(v)}\D v+1\right)\frac{\E^{I_{x_0}(u)}}{\gamma_{x_0}(u)}\D u\,<\,\infty\end{equation*} (recall that $\varphi(1)>0$). Denote this constant again by $\Lambda$.
	Analogously as in the proof of Theorem \ref{tm:TV}, let  $$\bar{\V}(r)\,:=\,\frac{1}{\Lambda}\int_{r_0}^{r}\E^{-I_{x_0}(u)}\int_u^{\infty}\left(\int_{r_0}^{v}\E^{-I_{x_0}(w)}\D w+1\right)\frac{\E^{I_{x_0}(v)}}{\gamma_{x_0}(v)}\D v\,\D u\,,\qquad r\geq r_0\,,$$  
	and, for arbitrary but fixed $r_1>r_0$,  let $\V:\R^{d}\to[0,\infty)$, $\V\in C^2(\R^d)$, be such that $\V(x)=\bar{\V}(|x-x_0|)+1$ for $x\in\R^d,$ $|x-x_0|\ge r_1.$ Then, for all $x\in\R^d$, $|x-x_0|\ge r_1$, it holds that \begin{equation}\label{exp}
	\mathcal{L}\V(x)\,\le\,-\frac{1}{2\Lambda}\V(x)\,,\end{equation}
	which is exactly 	the Lyapunov equation on 
	\cite[page 529]{Meyn-Tweedie-AdvAP-III-1993}
	with $c=1/2\Lambda$, $f(x)=\V(x)$, $C=\bar B_{r_1}(x_0)$  and $b=\sup_{x\in C}|\mathcal{L}V(x)|$.  The fact that $C$ is a petite set follows from	\cite[Theorems 5.1 and 7.1]{Tweedie-1994}, together with open-set irreducibility and $C_b$-Feller property of $\process{X}$.  
	Next, from 
	\cite[Proposition 6.1]{Meyn-Tweedie-AdvAP-II-1993}, \cite[Theorem 4.2]{Meyn-Tweedie-AdvAP-III-1993} and 
	aperiodicity it follows now that the are a petite set $\mathcal{C}\in\mathcal{B}(\R^d)$, 	 $T>0$ and a non-trivial measure $\nu_\mathcal{C}$ on $\mathcal{B}(\R^d)$, such that $\nu_\mathcal{C}(\mathcal{C})>0$ and $$p(t,x,B)\,\ge\,\nu_\mathcal{C}(B)\,,\qquad x\in \mathcal{C}\,,\ t\ge T\,,\ B\in\mathcal{B}(\R^d)\,.$$ In particular, 
	$$p(t,x,\mathcal{C})\,>\,0\,,\qquad x\in \mathcal{C}\,,\ t\ge T\,,$$	
	which is exactly the definition of aperiodicity used on \cite[page 1675]{Down-Meyn-Tweedie-1995}. Finally, observe that \eqref{exp} is also  the Lyapunov equation used on   \cite[page 1679]{Down-Meyn-Tweedie-1995} with $c=1/2\Lambda$,  $C=\bar B_{r_1}(x_0)$  and $b=\sup_{x\in C}|\mathcal{L}V(x)|$. The assertion now follows from \cite[Theorem 5.2]{Down-Meyn-Tweedie-1995}.
\end{proof}

Observe that in the proof of Theorem \ref{tm:TV} we did not use the fact that $\process{X}$ is a unique strong solution to \eqref{eq1}. All that we needed is that the martingale problem for $(b,c)$ is well posed, which is equivalent to that \eqref{eq1} admits a unique (in distribution) weak solution (see \cite[Theorem V.20.1]{Rogers-Williams-Book-II-2000}). According to \cite[Theorem 7.3.8]{Durrett-Book-1996} and \cite[Theorem V.24.1]{Rogers-Williams-Book-II-2000} the conclusion of Theorem \ref{tm:TV} remains true if, in addition to 
$\textbf{(C1)-(C3)}$,  $c(x)$ is Lipschitz continuous and there are $\Gamma>0$ and $\gamma\ge1$ such that \begin{equation}\label{eq:1}
\gamma^{-1}|y|^2\,\le\,\langle y,c(x)y\rangle\,\le\, \gamma|y|^2\quad\text{and}\quad |b(x)|^{2}+\lVert c(x)\rVert^{2}_{{\rm HS}}\,\le\, \Gamma(1+|x|^2)\,,\qquad x,\,y\in\R^{d}\,.\end{equation}
Moreover, under the above assumptions, \cite[Theorem V.24.1]{Rogers-Williams-Book-II-2000} states that $\process{X}$ is a Feller and strong Feller process. Recall, strong Feller property means that the corresponding semigroup maps $B_b(\R^d)$ to $C_b(\R^{d})$.
Also,  \eqref{eq:1}, together with $\textbf{(C1)-(C3)}$ and Lipschitz continuity of $c(x)$, implies open-set irreducibility and aperiodicity of $\process{X}$
(see \cite[Remark 4.3]{Stramer-Tweedie-1997}).

In the following theorem we discuss open-set irreducibility and aperiodicity of $\process{X}$ in the situation when  $c(x)$ is not necessarily Lipschitz continuous and uniformly elliptic.

\begin{theorem}\label{tm:IRAP} Assume $\textbf{(C1)-(C3)}$. Further, assume that there are $x_0\in\R^d$ and $r_0>0$, such that 
	\begin{itemize}
		\item [(i)]  there are  $\delta,\Gamma,\gamma>0$,  such that for all   $x,y\in B_{r_0}(x_0)$ we have that
		$$|b(x)-b(y)|+\lVert c(x)-c(y)\rVert_{{\rm HS}}\,\le\, \Gamma|x-y|^{\delta}\quad\textrm{and}\quad \langle y,c(x)y\rangle\,\ge\, \gamma|y|^{2}\,;$$
		\item[(ii)] $\mathbb{P}^x(\tau_{B_{r_0}(x_0)}<\infty)>0$ for all $x\in\R^d$, where $\tau_{B}:=\inf\{t\ge0: X_t\in B\}$ is the first hitting time of a set $B\subseteq\R^d.$ 
	\end{itemize}
	Then, $\process{X}$ is open-set irreducible and aperiodic.
\end{theorem}
\begin{proof}
	
	Due to \cite[Theorems 7.3.6 and 7.3.7]{Durrett-Book-1996} there is a strictly positive function  $q(t,x,y)$ on $(0,\infty)\times \bar{B}_{r_0}(x_0)\times\bar{B}_{r_0}(x_0)$, jointly continuous in $t$, $x$ and $y$, and twice continuously differentiable in $x$ on $B_{r_0}(x_0)$, satisfying \begin{equation*} \mathbb{E}^x(f(X_t),\tau_{\bar{B}^c_{r_0}(x_0)}>t)\,=\,\int_{B_{r_0}(x_0)}q(t,x,y)f(y)\,\D y\,,\qquad t>0,\, x\in B_{r_0}(x_0),\, f\in C_b(\R^d)\,,\end{equation*} where $\tau_{\bar{B}^c_{r_0}(x_0)}:=\inf\{t\ge0:X_t\in \bar{B}^c_{r_0}(x_0)\}.$ Clearly, by employing dominated convergence theorem, the above relation holds also for $\mathbbm{1}_O$, for any open set $O\subseteq B_{r_0}(x_0)$. Denote by $\mathcal{D}$ the class of all  $B\in\mathcal{B}(B_{r_0}(x_0))$ (the Borel $\sigma$-algebra on $B_{r_0}(x_0)$) such that \begin{equation*} \mathbb{P}^x(X_t\in B,\ \tau_{\bar{B}^c_{r_0}(x_0)}>t)\,=\,\int_{B}q(t,x,y)\, \D y\,,\qquad t>0,\, x\in B_{r_0}(x_0)\,.\end{equation*} Clearly, $\mathcal{D}$ contains the $\pi$-system of open rectangles in $B_{r_0}(x_0)$, and forms a $\lambda$-system. Hence, by employing  Dynkin's $\pi$-$\lambda$ theorem we conclude that $\mathcal{D}=\mathcal{B}(B_{r_0}(x_0)).$ Consequently, for any $t>0$, $x\in B_{r_0}(x_0)$ and $B\in\mathcal{B}(\R^d)$ we have that
	\begin{equation*} p(t,x,B)\,\ge\, \int_{B\cap B_{r_0}(x_0)}q(t,x,y)\, \D y\,.\end{equation*} 
	Set now $\phi(\cdot):=\lambda(\cdot\cap B_{r_0}(x_0))$, where $\lambda$ stands for the Lebesgue measure on $\R^d$. Then, $\phi$ is a $\sigma$-finite measure whose support has a non-empty interior.

	Let us now show that $\process{X}$ is $\phi$-irreducible.  Let  $x\in B^c_{r_0}(x_0)$ (for $x\in B_{r_0}(x_0)$ the assertion is obvious) and $B\in\mathcal{B}(\R^d)$, $\phi(B)>0$, be arbitrary. For all $s>0$ we have
	\begin{align*}
	\int_0^\infty p(t,x,B)\, \D t &\, \ge\, \int_s^\infty  p(t,x,B)\,\D t \\&\,=\, \int_s^\infty \int_{\R^d} p(t-s,x,\D y)p(s,y,B) \,\D t \\
	&\,\geq\, \int_s^\infty \int_{B_{r_0}(x_0)} p(t-s,x,\D y)p(s,y,B) \,\D t \\
	&\,=\, \int_{B_{r_0}(x_0)}p(s,y,B) \int_s^\infty p(t-s,x,\D y)\, \D t \,.
	\end{align*}
	The assertion now follows from the fact that $p(s,y,B)>0$ for $y\in B_{r_0}(x_0)$, and 
	\begin{equation*} \int_s^\infty p(t-s,x,B_{r_0}(x_0))\,\D t
	\,=\,\int_0^\infty p(t,x,B_{r_0}(x_0)) \,\D t
	\,=\,\mathbb{E}^x\left[\int_0^\infty \mathbbm{1}_{\{X_t\in B_{r_0}(x_0)\}}\,\D t\right]\,>\,0\,,
	\end{equation*}
	since $\process{X}$ has continuous sample paths, $B_{r_0}(x_0)$ is an open set and, by assumption, $\mathbb{P}^x(\tau_{B_{r_0}(x_0)}<\infty)>0$ for every $x\in\R^d$.
	
	Finally, let 
	us prove that $\process{X}$ is aperiodic. We show that 
	$$ \sum_{n=1}^\infty p(n,x,B)\,>\,0\,,\qquad x\in\R^d\,,$$ whenever $\phi(B)>0,$  $B\in \mathcal{B}(\R^d)$.
	Again, for $x\in B_{r_0}(x_0)$ the relation obviously holds. For $x\in B^c_{r_0}(x_0)$ and $B\in\mathcal{B}(\R^d)$, $\phi(B)>0$, we have that
	\begin{equation*}
	\sum_{n=1}^\infty p(n,x,B)\,\ge\,\int_{B_{r_0}(x_0)}\sum_{n=1}^\infty p(n-t,x,\D y)\,p(t,y,B)\,,\qquad t\in(0,1)\,.
	\end{equation*} Since $p(t,y,B)>0$ for $y\in B_{r_0}(x_0)$, it suffices to show that \begin{equation*}\sum_{n=1}^\infty p(n-t,x,B_{r_0}(x_0))\,\ge\,\mathbb{P}^x\left(\bigcup_{n=1}^\infty\{X_{n-t}\in B_{r_0}(x_0)\}\right)\,>\,0\end{equation*} for some $t\in(0,1)$. Assume this is not the case, that is, \begin{equation*} \mathbb{P}^x\left(\bigcup_{n=1}^\infty\{X_{n-t}\in B_{r_0}(x_0)\}\right)\,=\,0\,,\qquad t\in(0,1)\,.\end{equation*} This, in particular, implies that \begin{equation*}\mathbb{P}^x\left(\bigcup_{q\in\mathbb{Q}_+\setminus\ZZ_+}\{X_{q}\in B_{r_0}(x_0)\}\right)\,=\,0\,,\end{equation*} which is impossible since $\process{X}$ has continuous sample paths, $B_{r_0}(x_0)$ is an open set and $\mathbb{P}^x(\tau_{B_{r_0}(x_0)}<\infty)>0$ for every $x\in\R^d$. Thus, \begin{equation*} 
	\sum_{n=1}^\infty p(n,x,B)\,>\,0\,,\qquad x\in\R^d\,,\end{equation*} whenever $\phi(B)>0$, which concludes the proof.
\end{proof}

In the following proposition we give a sufficient condition for the second assumption in Theorem \ref{tm:IRAP} to hold.

\begin{proposition}\label{p:IRAP} Assume $\textbf{(C1)-(C3)}$. Then for any $x_0\in\R^d$ and $r_0>0$, provided that $c(x)$ is positive definite for all $x\in\R^d$, $|x-x_0|\geq r_0$, it holds that 
	$$\mathbb{P}^x(\tau_{B_{r_0}(x_0)}<\infty)\,>\,0\,, \qquad  x\in\R^d\,.$$
\end{proposition}

\begin{proof} 
	Let $0<\varepsilon<r_0$, and let  $$\bar{\V}(r):=\int_{r_0-\varepsilon}^{r}\E^{-I_{x_0}(u)}\D u\,,\qquad r\geq r_0-\varepsilon\,.$$ 
	Then, for $r> r_0-\varepsilon$ we have
	$$
	\bar{\V}'(r)\,=\,\E^{-I_{x_0}(r)}\,>\,0\qquad \text{and}\qquad 
	\bar{\V}''(r)\,=\,-\frac{\bar{\V}'(r)}{r}\iota_{x_0}(r)\,.
	$$		
	Further, let $\V:\R^{d}\to[0,\infty)$, $\V\in C^2(\R^d)$, be such that $\V(x)=\bar{\V}(|x-x_0|)$ for $x\in\R^d$, $|x-x_0|\ge r_0.$ Now, for $x\in\R^{d}$, $|x-x_0|\ge r_0$, we have
	\begin{align*}
	2\mathcal{L}\V(x)&\,=\,C_{x_0}(x)\bar{\V}''(|x-x_0|)+\frac{\bar{\V}'(|x-x_0|)}{|x-x_0|}(2A(x)-C_{x_0}(x)+2B_{x_0}(x))\\
	&\,=\, \frac{\bar{\V}'(|x-x_0|)}{|x-x_0|}\left(2A(x)-C_{x_0}(x)+2B_{x_0}(x)-C_{x_0}(x)\iota(|x-x_0|)\right)\\
	&\,\le\,0\,.
	\end{align*}
	Further, as we have already discussed,  for every $x\in\R^d$  the process $$\V(X_t)-\V(X_0)-\int_0^t\mathcal{L}\V(X_s)\,\D s\,,\qquad t\geq0\,,$$ is a local $\Prob^x$-martingale.	
	For $n\in\N$,  define 
	$\tau_n:=\tau_{B^c_n(x_0)}.$ Clearly, $\tau_n$, $n\in\N$, are stopping times such that (due to non-explosivity of $\process{X}$) $\tau_n\to\infty$ $\Prob^x$-a.s.  as $n\to\infty$ for all $x\in\R^d$.
	Hence, the processes $$\V(X_{t\wedge\tau_n})-\V(X_0)-\int_0^{t\wedge\tau_n}\mathcal{L}\V(X_s)\,\D s\,,\qquad t\geq0,\, n\in\N\,,$$ are $\Prob^x$-martingales. Now, for $x\in\R^d$, $|x-x_0|\ge r_0$, we have
	\begin{align*}
	2\mathbb{E}^x[\bar{\V}(|X_{t\wedge\tau_n\wedge\tau_{B_{r_0}(x_0)}}-x_0|)]-2\bar{\V}(|x-x_0|)&\,=\,2\mathbb{E}^x[\V(X_{t\wedge\tau_n\wedge\tau_{B_{r_0}(x_0)}})]-2\mathbb{E}^x[\V(X_0)]\\
	&\,=\,\mathbb{E}^x\int_0^{t\wedge\tau_n\wedge\tau_{B_{r_0}(x_0)}}2\mathcal{L}\V(X_s)\,\D s\\
	&\,\leq\, 0\,,
	\end{align*}
	that is,
	$$\mathbb{E}^x[\bar{\V}(|X_{t\wedge\tau_n\wedge\tau_{B_{r_0}(x_0)}}-x_0|)]\,\le\, \bar{\V}(|x-x_0|)\,.$$
	Thus,
	$$\mathbb{E}^x[\bar{\V}(|X_{t\wedge\tau_n}-x_0|)\Ind_{\{\tau_{B_{r_0}(x_0)}>\tau_n\}}]\,\le\, \bar{\V}(|x-x_0|)\,,\qquad x\in\R^d,\,\ |x-x_0|\ge r_0\,.$$ By letting $t\to\infty$ Fatou's lemma implies 
	$$\bar{\V}(n)\mathbb{P}^x(\tau_{B_{r_0}(x_0)}>\tau_n)\,\le\, \bar{\V}(|x-x_0|)\,,\qquad x\in\R^d,\, |x-x_0|\ge r_0\,.$$ Consequently, by letting $n\to\infty$, we conclude 
	$$ \mathbb{P}^x(\tau_{B_{r_0}(x_0)}=\infty)\,\le\, \frac{\bar{\V}(|x-x_0|)}{\bar \V(\infty)}\,<\,1\,,\qquad x\in\R^d,\, |x-x_0|\ge r_0\,,$$ that is, $\mathbb{P}^x(\tau_{B_{r_0}(x_0)}<\infty)>0$ for all $x\in\R^d.$
\end{proof}

As we have already commented, in   \cite[Theorem 5.4]{Douc-Fort-Guilin-2009}, \cite[page 1581]{Fort-Roberts-2005},  \cite[Theorem 1.30]{Kulik-Book-2015},
\cite[Theorem 3.3.6]{Kulik-Book-2018}, \cite[Theorem 3.3 (iv)]{Sandric-ESAIM-2016}, \cite[Theorem 6]{Veretennikov-1997} and \cite[Theorem 6]{Veretennikov-2000} it has been shown that a diffusion process $\process{X}$ (satisfying the assumptions from Corollary \ref{c:TV})   is sub-geometrically
ergodic with rate $t^{\alpha/(1-\alpha)}$, $0<\alpha<1$, if there are $\gamma>0$, $\Gamma>0$ and $r_0\ge0$, such that \eqref{eq:subgeo} holds true.
A simple example  which satisfies the relation in \eqref{eq2} but not the one in  \eqref{eq:subgeo} is the following. 
\begin{example}\label{e:TV}{\rm Let $\sigma(x)\equiv1$,  and let $b(x)$ be locally Lipschitz continuous and such that $b(x)=-{\rm sgn}(x)(\cos x+1)$ for all $|x|$ large enough, where $${\rm sgn}(x)\,:=\,\left\{
		\begin{array}{ll}
		1, & x\geq0\,, \\
		-1, & x<0\,.
		\end{array}
		\right.$$  Clearly, $b(x)$ and $\sigma(x)$ satisfy $\textbf{(C1)-(C3)}$ and  define, through \eqref{eq1}, an open-set irreducible and aperiodic diffusion process $\process{X}$.  The condition in \eqref{eq2} now reduces to showing 
		that there is $r_0\ge 0$  such that
		$$\int_{r_0}^\infty\left(\int_{r_0}^u\E^{2\sin v +2v}+1\right)^\alpha\E^{-2\sin u -2u}\D u\,<\,\infty\,,$$
		which  can be obviously obtained  for any $0<\alpha<1.$ On the other hand,  the condition in \eqref{eq:subgeo} is equivalent to showing that there are $\gamma>0$, $\Gamma>0$ and $r_0\ge 0$, such that $$
		\frac{\gamma-1}{2}-x\,{\rm sgn}(x)(\cos x+1)\,\leq\,-\Gamma|x|^{\gamma\alpha-\gamma+2}\,,\qquad |x|\geq r_0\,.
		$$ However, observe that in the points of the form $x=(2k+1)\pi$, $k\in\ZZ$, the second term on the left-hand side in the above inequality vanishes. Thus, we conclude that it is necessary that $0<\gamma<1$ and $\gamma\alpha-\gamma+2<0$, which is impossible. Note also that if we take $b(x)$ to be locally Lipschitz continuous and such that $b(x)=-{\rm sgn}(x)(\cos x+\varrho)$ for all $|x|$ large enough, where $\varrho>0$, then we again easily conclude that \eqref{eq2} holds for any $0<\alpha<1$. On the other hand, by the same reasoning as above, \eqref{eq:subgeo} can never hold. Observe that for $0<\varrho<1$ the drift function generates a region  in which the process is ``pushed towards infinity" (set of points for which ${\rm sgn}(x)b(x)>0$). The condition in \eqref{eq2} says that this region is small compared to the region  in which the process is ``pushed towards the center of the state space" (set of points for which ${\rm sgn}(x)b(x)<0$) and which is responsible for the ergodic behavior.}
\end{example}

\begin{proposition}\label{p1} Assume $\textbf{(C1)-(C3)}$. Further, assume that
	$\gamma<2/(1-\alpha)$ and there are $r_0\ge0$ and $\Delta\ge1$, such that $\Delta^{-1}\le C_0(x)\le \Delta$  for all $|x|\geq r_0$. Then, \eqref{eq2} (with $x_0=0$) is a consequence of \eqref{eq:subgeo}.
\end{proposition}
\begin{proof}
	We have that
	$$\iota_0(r)\,=\,\sup_{|x|=r}\frac{2\left( A(x)-\left(1-\frac{\gamma}{2} \right) C_0(x)+B_0(x)\right)+(1-\gamma)C_0(x)}{C_0(x)}\,\leq\,-\frac{2\Gamma}{\Delta}r^{\gamma\alpha-\gamma+2}+1-\gamma$$ for all $r\geq r_1,$ for some $r_1\geq r_0$ large enough. Thus,
	there are $\Gamma_1>0$ and $r_2\geq r_1$, such that
	$$\iota_0(r)\,\leq\, -\Gamma_1r^{\gamma\alpha-\gamma+2}\,,\qquad r\geq r_2\,.$$ This automatically implies that there are $\Gamma_2>0$ and $r_3\geq r_2$, such that
	$$I_0(r)\,\leq\, -\Gamma_2r^{\gamma\alpha-\gamma+2}\,,\qquad r\geq r_3\,.$$
	Now, 
	by employing L'Hospital's rule (here we use the assumption $\gamma<2/(1-\alpha)$), we have that
	$$\lim_{u\to\infty}\frac{\left(\int_{r_3}^u\E^{-I_0(v)}\D v+1\right)}{\E^{-I_0(u)}}\,=\,0\,.$$ Hence, there is  $r_4\geq r_3$ such that $$\int_{r_3}^u\E^{-I_0(v)}\D v+1\,\leq\, \E^{-I_0(u)}\qquad  u\geq r_4\,.$$
	Finally, we conclude
	$$\int_{r_4}^\infty\left(\int_{r_4}^u\E^{-I_0(v)}\D v+1\right)^\alpha\E^{I_0(u)}\D u\,\le\,\int_{r_4}^\infty\E^{(1-\alpha)I_0(u)}\D u\,<\,\infty\,,$$ which proves the assertion.
\end{proof}

In the following proposition,  which  generalizes \cite[Lemma 1.2]{Chen-2000} to the sub-geometric case, we give sufficient conditions ensuring \eqref{eq2}.

\begin{proposition}\label{lm:SUB} 
	Let  $c\geq0$, and let $\rho(t)$ be a non-negative and non-decreasing differentiable function defined on $[0,\infty)$.
	Further, let $f(r)$ and $g(r)$  be non-negative Borel measurable functions, also defined on $[0,\infty)$, satisfying 
	\begin{equation}\label{eq3}\Delta\,:=\,\sup_{r\geq r_0} \rho\left(\int_{r_0}^{r}\,g(u)\D u+c\right)^{1+\beta}\int_r^{\infty}f(u)\,\D u\,<\,\infty\end{equation} for some $r_0\geq0$ and $\beta\ge0$. Then,
	\begin{itemize}
		\item [(i)] if $\beta>0$, $$\int_r^{\infty}\rho\left(\int_{r_0}^{u}g(v)\,\D v+c\right)f(u)\,\D u\,\le\, \frac{\Delta(1+\beta)}{\beta}\rho\left(\int_{r_0}^{r}g(u)\,\D u+c\right)^{-\beta}\,,\qquad r\geq r_0\,.$$
		\item[(ii)] if $\beta=0$, and $\int_{r_0}^{\infty} g(r)\,\D r<\infty$ or  $\rho(t)$ is bounded,
		$$	\int_r^{\infty}\rho\left(\int_{r_0}^{u}g(v)\,\D v+c\right)f(u)\,\D u\,\le\,\Delta+\Delta\ln\frac{\rho\left(\int_{r_0}^{\infty}g(u)\,\D u+c\right)}{\rho\left(\int_{r_0}^{r}g(u)\,\D u+c\right)}\,,\qquad r\ge r_0\,.$$
	\end{itemize} 
\end{proposition}
\begin{proof}
	Set
	$F(r)=\int_r^{\infty}f(u)\,\D u,$ $r\geq r_0$. Then, by assumption, $$F(r)\,\le\, \Delta\,\rho\left(\int_{r_0}^{r}g(u)\,\D u+c\right)^{-1-\beta}\,,\qquad r\geq r_0\,.$$
	Consequently, for $r\geq r_0$, we have that
	\begin{align*}
	&\int_r^{\infty}\rho\left(\int_{r_0}^{u}g(v)\,\D v+c\right)f(u)\,\D u\\
	&\,=\,-\int_r^{\infty}\rho\left(\int_{r_0}^{u}g(v)\,\D v+c\right)\D F(u)\\
	&\,\le\, \rho\left(\int_{r_0}^{r}g(u)\,\D u+c\right)F(r) +\int_r^{\infty}\rho'\left(\int_{r_0}^{u}g(v)\,\D v+c\right)g(u)F(u)\,\D u\\
	&\,\le\, \Delta\,\rho\left(\int_{r_0}^{r}g(u)\,\D u+c\right)^{-\beta} +\Delta\int_r^{\infty}\rho'\left(\int_{r_0}^{u}g(v)\,\D v+c\right)g(u)\rho\left(\int_{r_0}^{u}g(v)\,\D v+c\right)^{-1-\beta}\D u\,.\end{align*}
	Now, under the assumption in (i) we have that
	\begin{align*}
	&\int_r^{\infty}\rho\left(\int_{r_0}^{u}g(v)\,\D v+c\right)f(u)\,\D u\\
	&\,\le\,\Delta\,\rho\left(\int_{r_0}^{r}g(u)\,\D u+c\right)^{-\beta} -\frac{\Delta}{\beta}\int_r^{\infty}\D \rho\left(\int_{r_0}^{u}g(v)\,\D v+c\right)^{-\beta}\\
	&\,\le\, \Delta\,\rho\left(\int_{r_0}^{r}g(u)\,\D u+c\right)^{-\beta}+\frac{\Delta}{\beta}\rho\left(\int_{r_0}^{r}g(u)\,\D u+c\right)^{-\beta}\\
	&\,=\,\frac{\Delta(1+\beta)}{\beta}\rho\left(\int_{r_0}^{r}g(u)\,\D u+c\right)^{-\beta},
	\end{align*} 
	where in the second step we employed integration by parts formula. On the other hand, under the assumptions in (ii), 
	\begin{align*}
	\int_r^{\infty}\rho\left(\int_{r_0}^{u}g(v)\,\D v+c\right)f(u)\,\D u&\,\le\,\Delta+\Delta\int_r^{\infty}d\ln\left(\rho\left(\int_{r_0}^{u}g(v)\,\D v+c\right)\right)\\
	&\,=\,\Delta+\Delta\ln\frac{\rho\left(\int_{r_0}^{\infty}g(u)\,\D u+c\right)}{\rho\left(\int_{r_0}^{r}g(u)\,\D u+c\right)}\,,
	\end{align*}
	which concludes the proof.
\end{proof}
As a direct consequence of the proposition we see that \eqref{eq2} holds true if $$\sup_{r\geq r_0} \varphi\left(\int_{r_0}^{r}\E^{-I_{x_0}(u)}\D u+1\right)^{1+\beta}\int_r^{\infty}\frac{\E^{I_{x_0}(u)}}{\gamma_{x_0}(u)}\D u\,<\,\infty$$ for some $\beta>0.$

\subsection{Ergodicity of Markov processes with jumps}\label{du}

In this subsection, as an application of  Theorem \ref{tm:TV}, we  discuss sub-geometric ergodicity of a class of Markov processes with jumps.
First, we consider jump-diffusion   processes generated by  operator of the form
\begin{align}\label{eq7}\mathcal{L}f(x)&\,=\,\langle b(x),\nabla f(x)\rangle+\frac{1}{2}{\rm Tr}\,c(x)\nabla^2f(x)\nonumber\\
&\ \ \ \ +\int_{\R^d}\left(f(y+x)-f(x)-\langle y, \nabla f(x)\rangle\Ind_{B_1(0)}(y)\right)\nu(x,\D y)\,,\end{align} where
$b(x)$ is an $\R^d$-valued Borel
measurable function, $c(x)$  is a symmetric non-negative definite $d\times d$ matrix-valued Borel measurable function, and $\nu(x,\D y)$ is a non-negative Borel kernel on 
$(\R^d,\mathcal{B}(\R^d))$, called the L\'evy kernel,  satisfying
$$\nu(x,\{0\})\,=\,0\,,\quad\text{and}\quad\int_{\R^d}(1\wedge|y|^{2})\nu(x,\D y)\,<\,\infty\,,\qquad x\in\R^d\,.$$  Clearly, if $\nu(x,\D y)$ is a null-measure, then $\mathcal{L}$ becomes a diffusion operator. In the sequel, we assume that 
\begin{description}
	\item [(A1)] there is a  c\`adl\`ag Markov process $(\Omega,\mathcal{F}, \process{\mathcal{F}},\process{\theta},$ $\process{X},\{\Prob^x\}_{x\in\R^d})$, denoted by  $\process{X}$ in the sequel, which we call jump-diffusion process, such that
	for every $f \in C^2(\R^d)$ the process
	$$f(X_t)-f(X_0)-\int_0^{t}\mathcal{L}f(X_s)\,\D s\,,\qquad t\geq0\,,$$ is a $\mathbb{P}^x$ local martingale for all $x\in\R^d$ under the natural filtration;
	\item [(A2)] the  process $\process{X}$    satisfies the $C_b$-Feller property;
	\item [(A3)] the process $\process{X}$ is open-set irreducible and aperiodic.
\end{description}
Here,  $C_b^2(\R^d)$ denotes the space of twice continuously differentiable functions with bounded derivatives.
Let us remark that $\textbf{(A1)}$ always  holds for the infinitesimal generator
$(\mathcal{A},\mathcal{D}_\mathcal{A})$ of $\process{X}$ (see   \cite[Theorem 2.2.13 and Proposition 4.1.7]{Ethier-Kurtz-Book-1986}).  
We refer the readers to \cite{Bottcher-Schilling-Wang-2013} for conditions, in terms of   $b(x),$ $c(x) $ and $\nu(x,\D y)$, ensuring $\textbf{(A1)}$ and $\textbf{(A2)}$. 
Open-set irreducibility and  aperiodicity of jump-diffusion  processes is a very well-studied topic in the literature. In particular, we refer the readers to  \cite{Kolokoltsov-2000} and \cite{Kolokoltsov-Book-2011} for the case of  so-called stable-like processes,
to \cite{Knopova-Schilling-2012}, \cite{Knopova-Schilling-2013}, \cite{Kwoon-Lee-1999}, \cite[Remark 3.3]{Pang-Sandric-2016}  \cite[Theorem 2.6]{Sandric-TAMS-2016} and \cite{Stroock-1975}   for the case of jump-diffusion  processes with bounded coefficients, and to \cite{Arapostathis-Pang-Sandric-2019},  \cite{Bass-Cranston-1986},  \cite{Ishikawa-2001}, \cite{Knopova-Kulik-2014},  \cite{Masuda-2007, Masuda-Erratum-2009}  and \cite{Picard-1996, Picard-Erratum-2010} for the case of a class of jump-diffusion processes obtained as a solution to certain jump-type SDEs.
According to \cite[Theorem 3.2]{Tweedie-1994}, $\process{X}$ will be open-set irreducible and aperiodic if it is strong Feller (actually it suffices to assume that $\process{X}$ is a T-model in the sense of \cite{Tweedie-1994}, which is a certain weak version of the strong Feller property) and $\Prob^x(X_t\in O)>0$ for every $t>0$, $x\in\R^d$ and non-empty open set $O\subseteq\R^d$. If $b(x)$ is continuous and bounded, $c(x)$  continuous, bounded and positive definite,  $x\mapsto\int_{B}(1\wedge|y|^2)\nu(x,\D y)$  continuous and bounded  for any $B\in\mathcal{B}(\R^d)$, and 
\begin{equation*}(x,\xi)\,\mapsto\,
i\langle \xi,b(x)\rangle + \frac{1}{2}\langle\xi,c(x)\xi\rangle +
\int_{\R^{d}}\left(1-e^{i\langle\xi,y\rangle}+i\langle\xi,y\rangle\Ind_{B_1(0)}(y)\right)\nu(x,\D y)\end{equation*} continuous,	
then 
\begin{itemize}
	\item [(i)] there is a unique non-explosive strong Markov process $\process{X}$ with infinitesimal generator $(\mathcal{A},\mathcal{D}_\mathcal{A})$ such that $C_c^\infty(\R^d)\subseteq
	\mathcal{D}_\mathcal{A}$, and $\mathcal{A}|_{C_c^\infty(\R^d)}$ takes the form in \eqref{eq7}, where 
	$C_c^\infty(\R^d)$ stands for the space of smooth functions with compact support;
	\item[(ii)]the operator $\mathcal{L}:=\mathcal{A}|_{C_c^\infty(\R^d)}$ satisfies \textbf{(A1)};
	\item[(iii)] the  semigroup of $\process{X}$ satisfies the Feller and strong Feller property\,,
\end{itemize} 
(see \cite[Theorems 2.37, 3.23, 3.24, 3.25]{Bottcher-Schilling-Wang-2013} and \cite[Remark after Theorem 4.3]{Stroock-1975}). 
Finally, we also assume
\begin{description}
	\item[(A4)] there is $\rho>0$  such that  $\nu(x,B^c_{|x|}(-x))=0$ and $\int_{B_1(0)}|y|\,\nu(x,\D y)<\infty$ for all $x\in\R^d$, $|x|\ge \rho$;
	\item[(A5)]  the functions $b(x)$, $c(x)$ and $x\mapsto\int_{B_1(0)}y\,\nu(x,\D y)$  are continuous on $B^c_{\rho}(0)$.
\end{description}
Assumption \textbf{(A4)} means that when $\process{X}$ is far away from the center of the state space, it admits  bounded jumps only, with maximal intensity equal twice the distance to the origin. Also, with each jump, it comes closer to the center of the state space. 

In the following theorem we give sufficient conditions for sub-geometric ergodicity of a class of jump-diffusion processes satisfying \textbf{(A1)-(A5)}. We use the same notation as in Theorem \ref{tm:TV}, with $$B_{x_0}(x)\,:=\,\bigl\langle x-x_0,b(x)-\int_{B_1(0)}y\,\nu(x,\D y)\bigr\rangle\,,\qquad x\in\R^d\,.$$

\begin{theorem}\label{tm:TVJD}Let $\process{X}$ be an open-set irreducible and aperiodic jump-diffusion process with coefficients $b(x)$, $c(x)$ and $\nu(x,\D y)$, satisfying \textbf{(A1)-(A5)}.
	Further, let $\varphi:[1,\infty)\longrightarrow(0,\infty)$ be a  non-decreasing, differentiable and concave  function satisfying $\lim_{t\to\infty}\varphi'(t)=0$ and the relation in \eqref{eq2}
	for some $x_0\in\R^d$ and  $ r_0\ge \rho +|x_0|$,	and assume that $c(x)$ is positive definite for all $x\in\R^d$, $|x-x_0|\ge r_0$. 	Then, $\process{X}$ admits a unique invariant  $\pi\in\mathcal{P}$ such that
	$$\lim_{t\to\infty}\varphi(\varPhi^{-1}(t))\rVert\delta_x P_t-\pi\lVert_{{\rm TV}}\,=\,0\,,\qquad x\in\R^d\,,$$ where  $\Phi(t)$ is as in Theorem \ref{tm:TV}.
\end{theorem}
\begin{proof} We proceed  as in the proof of Theorem \ref{tm:TV}.
	Define $$\bar{\V}(r)\,:=\,\int_{r_0}^{r}\E^{- I_{x_0}(u)}\int_u^{\infty}\varphi_\Lambda\left(\int_{r_0}^{v}\E^{-I_{x_0}(w)}\D w+1\right)\frac{\E^{I_{x_0}(v)}}{\gamma_{x_0}(v)}\D v\,\D u\,,\qquad r\geq r_0\,,$$ where $\varphi_\Lambda(t)=\varphi(t)/\Lambda$. Clearly,
	\begin{equation}\label{e12}\bar{\V}(r)\,\le\,\int_{r_0}^r\E^{-I_{x_0}(u)}\D u\,,\qquad r\ge r_0\,,\end{equation} and, because of \textbf{(A5)}, $\bar{\V}(r)$ is twice continuously differentiable on $(r_0,\infty)$.  
	Further, for arbitrary, but fixed,  $r_1>r_0$ let  $\tilde{\V}:[0,\infty)\to[0,\infty)$ be  non-decreasing  on $[0,\infty)$,  $\tilde{\V}(r)=\bar{\V}(r)$ on $[r_1,\infty)$, and such that   $\V(x):=\tilde{\V}(|x-x_0|)+1$ is twice continuously differentiable on $\R^d$. Now, because of \textbf{(A1)} and \textbf{(A4)}, $\mathcal{L}\V(x)$ is well defined and the process $$\V(X_t)-\V(X_0)-\int_0^{t}\mathcal{L}\V(X_s)\,\D s\,\qquad t\geq0\,,$$ is a local martingale.  For  $x\in\R^d$, $|x|\geq r_1$, we have that
	\begin{align*}
	\mathcal{L}\V(x)\,=&\,\frac{1}{2}C_{x_0}(x)\bar{\V}''(|x-x_0|)+\frac{\bar{\V}'(|x-x_0|)}{2|x-x_0|}(2A(x)-C_{x_0}(x)+2\langle x-x_0,b(x)\rangle)\\&  +\int_{\R^d}\bigl(\V(y+x)-\V(x)-\langle y,\nabla\V(x)\rangle\Ind_{B_1(0)}(y)\bigr)\nu(x,\D y)\\
	\,\le&\,\frac{1}{2}C_{x_0}(x)\bar{\V}''(|x-x_0|)+\frac{\bar{\V}'(|x-x_0|)}{2|x-x_0|}(2A(x)-C_{x_0}(x)+2B_{x_0}(x))\\  \,\le&\,-\frac{1}{2}\varphi_\Lambda\left(\int_{r_0}^{|x-x_0|}\E^{-I_{x_0}(u)}\D u+1\right)\\
	\,\le&\, -\frac{1}{2}\varphi_\Lambda(\V(x))  \,,
	\end{align*}
	where in the second step we used  \textbf{(A4)} and properties of $\V(x)$ (that is, $\tilde{\V}(r)$), and the final  step follows from \eqref{e12}. 
	Finally, because of \textbf{(A2)} and \textbf{(A5)}, as in the proof of Theorem \ref{tm:TV}, we are again in a position to apply \cite[Theorems 3.2 and 3.4 (i)]{Douc-Fort-Guilin-2009} and \cite[Theorems 5.1 and 7.1]{Tweedie-1994}, which concludes the proof.
\end{proof}

Let us now give several remarks.
\begin{remark}\label{r2}{\rm
		
		\begin{itemize}
			\item [(a)]  If $2A(x)-C_{x_0}(x)+2B_{x_0}(x)\leq0$ for all $x\in\R^d$, $|x-x_0|\ge r_0$,  then  we can replace   $\gamma_{x_0}(r)$ and $\iota_{x_0}(r)$ by
			\begin{align*}
			\gamma_{x_0}(r)&\,=\,\inf_{|x-x_0|=r}N_{x_0}(x)\,,\qquad r>0\,,\\
			\iota_{x_0}(r)&\,=\,\sup_{|x-x_0|=r}\frac{2A(x)-C_{x_0}(x)+2B_{x_0}(x)}{N_{x_0}(x)}\,,\qquad r>0\,,\\
			\end{align*}
			where $$N_{x_0}(x)\,=\,\frac{\langle x-x_0,(c(x)+n(x))(x-x_0)\rangle}{|x-x_0|^2}\,,\qquad x\in\R^d\setminus\{0\}\,,$$ and $n(x)=(n_{ij}(x))_{i,j=1,\dots,d}$ with $n_{ij}(x)=\int_{B_1(0)}y_iy_j\nu(x,\D y)$. Also, in this situation, the requirement in Theorem \ref{tm:TVJD} that $c(x)$ is positive definite for all $x\in\R^d$, $|x-x_0|\ge r_0$, can be replaced by the requirement that $c(x)+n(x)$ is positive definite for all $x\in\R^d$, $|x-x_0|\ge r_0$.
			\item [(b)] If $\varphi(t)$ is bounded, then 
			\eqref{eq2} reads $$\int_{r_0}^\infty\frac{\E^{I_{x_0}(u)}}{\gamma_{x_0}(u)}\D u\,<\,\infty\,,$$ and 			
			gives a condition for ergodicity (see   \cite[Theorem 1.2]{Wang-2008} for the one-dimensional case).
			\item [(c)]  If in Theorem \ref{tm:TVJD} $\liminf_{t\to\infty}\varphi'(t)>0$ then, as in Proposition \ref{p:4}, 
			we conclude that $\process{X}$ is geometrically ergodic 
			(see also  \cite[Theorem 1.3]{Wang-2008} for the one-dimensional case).
		\end{itemize}
	}
\end{remark}

Let us now give an example satisfying conditions from  Theorem \ref{tm:TVJD}.

\begin{example}[L\'evy-driven SDEs]\label{e3}
	\rm{Let $\process{Y}$ be an $n$-dimensional L\'evy process, and let $\Phi:\R^{d}\to\R^{d\times n}$ be bounded and locally Lipschitz continuous. Then, in \cite[Theorems 3.1 and 3.5, and Corollary 3.3]{Schilling-Schnurr-2010} (see also \cite[Theorem 3.8]{Bottcher-Schilling-Wang-2013}) it has been shown that  the SDE \begin{equation}\label{SDE1}\D X_t\,=\,\Phi(X_{t-})\,\D Y_t\,,\qquad X_0=x\in\R^{d}\,,\end{equation} admits a unique strong solution which is a non-explosive strong Markov process whose semigroup satisfies the Feller and $C_b$-Feller property (thus \textbf{(A2)} holds true). Also, it has been shown that $\process{X}$ satisfies \textbf{(A1)} with certain coefficients $b(x)$, $c(x)$ and $\nu(x,\D y)$, which   in a special case we give below.
		Observe that the following SDE is a special case of \eqref{SDE1},
		\begin{equation}\label{SDE2}
		\D X_t\,=\,\Phi_1(X_{t-})\,\D t+\Phi_2(X_{t-})\,\D B_t\,+\Phi_3(X_{t-})\,\D Z_t\,,\qquad  X_0=x\in\R^d\,,
		\end{equation}
		where $\Phi_1:\R^{d}\to\R^{d}$, $\Phi_2:\R^{d}\to\R^{d\times p}$ and $\Phi_3:\R^{d}\to\R^{d\times q}$, with $p+q=n-1$, are locally Lipschitz continuous and bounded, $\process{B}$ is a $p$-dimensional Brownian motion, and $\process{Z}$ is a $q$-dimensional pure-jump L\'evy process (that is, a L\'evy process determined by a L\'evy triplet of the form $(0,0,\nu_Z(\D y))$) independent of $\process{B}$. Namely, set $\Phi(x)=\bigl(\Phi_1(x),\Phi_2(x),\Phi_3(x)\bigr)$, and $Y_t=(t,B_t,Z_t)^T$, $t\ge0$. Assume now that $d=p=q=1$. Then, from 
		\cite[Theorem 3.1]{Schilling-Schnurr-2010} we see that the corresponding coefficients read 
		\begin{align*}b(x)&\,=\,\left\{\begin{array}{cc}
		\Phi_1(x)\,, & \Phi_3(x)\,=\,0\,, \\
		\Phi_1(x)+\int_{\R}y\left(\Ind_{B_1(0)}(y)-\Ind_{B_{|\Phi_3(x)|}(0)}(y)\right)\nu_Z\left(\frac{\D y}{\Phi_3(x)}\right)\,,& \Phi_3(x)\,\neq\,0\,,
		\end{array}\right.\\
		c(x)&\,=\,\Phi_2^2(x)\\
		\nu(x,\D y)&\,=\,\left\{\begin{array}{cc}
		0\,, & \Phi_3(x)\,=\,0\,, \\
		\nu_Z\left(\frac{\D y}{\Phi_3(x)}\right)\,,& \Phi_3(x)\,\neq\,0\,.
		\end{array}\right.
		\end{align*}
		Take now, for simplicity, 
		$$\Phi_1(x)\,=\,\Phi_3(x)\,=\,\left\{\begin{array}{cc}
		-1\,, & x\,\ge\,1\,, \\
		-x\,,& -1\,\le\,x\,\leq\,1\,,\\
		1\,, & x\,\le\,-1\,,
		\end{array}\right.$$
		$\Phi_2(x)=1$, and 
		$\nu_Z(\D y)=f(y)\D y$ with $f(y)$ being the probability density function of the continuous uniform distribution on the segment $[0,1]$. It is straightforward to see that $\process{X}$ satisfies \textbf{(A4)} and  \textbf{(A5)}. Open-set irreducibility and aperiodicity of $\process{X}$ have been considered on \cite[page 43]{Masuda-2007} (see also \cite[Theorem 3.1]{Kwoon-Lee-1999}). Finally, since 
		\[\arraycolsep=1pt\def\arraystretch{1.3}
		B_0(x)\,=\,\left\{\begin{array}{cc}
		-\frac{1}{2}x\,, & x\,\ge\,1\,, \\
		\frac{1}{2}x\,,&  x\,\le\,-1\,,
		\end{array}\right.
		\]
		it is elementary to check that   
		$\process{X}$ satisfies \eqref{eq2} with $x_0=0$, $r_0=1$ and $\varphi(t)=t^\alpha$, $\alpha\in(0,1)$. Thus, $\process{X}$ is sub-geometrically ergodic with rate $t^{\alpha/(1-\alpha)}.$
		
		Observe that the same conclusion follows by employing a version of the relation in \eqref{eq:subgeo} including jumps (see \cite[Theorem 3.3]{Sandric-ESAIM-2016}). However, if  we take $\Phi_1(x)=-{\rm sgn}(x)(\cos x+3/2)$ (analogously as in Example \ref{e:TV}), then it is not hard to see that \eqref{eq:subgeo} does not hold. On the other hand, Theorem \ref{tm:TVJD} (with $x_0=0$, $r_0=1$ and $\varphi(t)=t^\alpha$, $\alpha\in(0,1)$)  implies that  	$\process{X}$ is again sub-geometrically ergodic with rate $t^{\alpha/(1-\alpha)}$.
}\end{example}

An alternative approach in obtaining a class of Markov processes with jumps (from diffusion processes) is through the Bochner's subordination method. Recall, a subordinator  $\process{S}$ is a non-decreasing L\'{e}vy process on $\left[0,\infty\right)$ with Laplace transform
$$\mathbb{E}\left[\E^{-uS_t}\right] \,=\, \E^{-t\phi(u)}\,, \qquad u>0,\, t\geq 0\,.$$
The characteristic (Laplace) exponent $\phi:(0,\infty)\to(0,\infty)$  is a Bernstein function, that is, it is of class $C^\infty$ and $(-1)^n\phi^{(n)}(u)\ge0$ for all $n\in\N$. It is well known that every Bernstein function  admits a unique (L\'{e}vy-Khintchine) representation 
$$\phi(u)\,=\,bu+\int_{(0,\infty)}(1-\E^{-uy})\,\nu(\D y)\,, \qquad u>0\,,$$
where $b\geq0$ is the drift parameter and $\nu$ is a L\'{e}vy measure, that is, a measure on $\mathcal{B}((0,\infty))$ satisfying $\int_{(0,\infty)}(1\wedge y)\nu(\D y)<\infty$.
For more on subordinators and Bernstein functions we refer the readers 	to the monograph \cite{Schilling-Song-Vondracek-Book-2012}.
Let now $\process{M}$ be a Markov process with state space $(\R^d,\mathcal{B}(\R^d))$ and transition kernel $p(t,x,\D y)$. Further, let $\process{S}$ be a subordinator with characteristic exponent $\phi(u)$, independent of   $\process{M}$. 
The process $M_t^{\phi}:=M_{S_t}$, $t\ge0$, obtained from $\process{M}$ by 
a random time change  through $\process{S}$, is referred to as the subordinate process $\process{M}$ with subordinator $\process{S}$ in the sense of Bochner. It is easy to see that $\process{M^\phi}$ is again a Markov process with  transition kernel
$$p^\phi(t,x,\D y)\,=\,\int_{\left[0,\infty\right)} p(s,x,\D y)\,\mu_t(\D s)\,,$$
where $\mu_t(\cdot)=\mathbb{P}(S_t\in\cdot)$ is the transition probability of $S_t$, $t\ge0$. Also, it is elementary to check that if $\pi$ is an invariant probability measure for $\process{M}$, then $\pi$ is also invariant for the subordinate process $\process{M^\phi}$. In \cite{Deng-Schilling-Song-2017} it has been shown that if  $\process{M}$ is sub-geometrically ergodic with Borel measurable rate $r(t)$ (with respect to the total variation distance), then $\process{M^\phi}$ is sub-geometrically ergodic with rate $r_\phi(t)=\mathbb{E}[r(S_t)].$  Therefore, as an direct application of Theorem \ref{tm:TV}, we obtain sub-geometric ergodicity results for a class of subordinate diffusion processes.

\section{Ergodicity with respect to Wasserstein distances}\label{s3}

In this section, we first prove Theorems  
\ref{tm:WASS1} and \ref{tm:WASS2}.
Then, we discuss  sub-geometric ergodicity of two classes of Markov processes with jumps. 

\subsection{Proof of Theorems \ref{tm:WASS1} and \ref{tm:WASS2}}

In Theorem \ref{tm:TV} we discussed sub-geometric ergodicity of a diffusion process $\process{X}$ (given through \eqref{eq1}) with respect to the total variation distance. 
Crucial  assumptions  in this result were open-set irreducibility and aperiodicity  of $\process{X}$. In order to ensure these properties the discussion after Proposition \ref{p:4}  and Theorem \ref{tm:IRAP} suggest that quite strong regularity and smoothness  assumptions  of the coefficient  $c(x)$ are needed.
By using a completely different approach to this problem, the so-called synchronous coupling method (see \cite[Example 2.16]{Chen-Book-2005} for details), we   derive sub-geometric ergodicity for a class of diffusions with (possibly) singular diffusion coefficient.

We start with the following auxiliary result, which will be crucial in the proofs of Theorems \ref{tm:WASS1} and \ref{tm:WASS2}, and which is a version of non-linear convex Gronwall's inequality.

\begin{lemma}\label{lm:WASS1} Let $\Gamma>0$, and  let $f:[0,T)\to[0,\infty)$, with $0<T\le\infty,$ and $\psi:[0,\infty)\to[0,\infty)$ be  
	such that
	\begin{itemize}
		\item [(i)] $f(t)$ is  absolutely continuous on $[t_0,t_1]$ for any $0<t_0<t_1<T$;
		\item[(ii)] $f'(t)\leq -\Gamma\,\psi(f(t))$ a.e. on $[0,T)$;
		\item[(iii)] $\psi(f(t))>0$ a.e. on $[0,T)$, and $\Psi_{f(0)}(t):=\int_t^{f(0)}\frac{\D s}{\psi(s)}<\infty$ for all $t\in(0,f(0)]$.
	\end{itemize}  
	Then, $$f(t)\,\leq\,\Psi_{f(0)}^{-1}(\Gamma t)\,,\qquad 0\le t<\Gamma^{-1}\Psi_{f(0)}(0)\wedge T\,.$$
	In addition, if 
	there is $\kappa\in[f(0),\infty]$ such that $\Psi_\kappa(t):=\int_t^{\kappa}\frac{\D s}{\psi(s)}<\infty$ for $t\in(0,\kappa]$,
	then $$f(t)\,\leq\,\Psi_\kappa^{-1}(\Gamma t)\,,\qquad 0\le t<\Gamma^{-1}\Psi_{f(0)}(0)\wedge T\,.$$  Also, if $\psi(t)$ is convex and vanishes at zero, then $\Psi_{f(0)}(0)=\infty$, that is, the above relations hold for all $t\in[0,T)$.
\end{lemma}
\begin{proof}
	By assumption, $$-\Psi_{f(0)}(f(t))\,=\,\int_{f(0)}^{f(t)}\frac{\D s}{\psi(s)}\,=\,\int_0^t\frac{f'(s)\,\D s}{\psi(f(s))}\,\le\,-\Gamma t\,,\qquad t\in[0,T)\,.$$ Now, the  first assertion follows. 
	
	The second claim follows from the fact that 
	$\Psi_{f(0)}(t)\le\Psi_\kappa(t)$ for all $t\in(0,f(0)]$, while 
	the last part follows from  $$\psi(t)\,=\,\psi(t+(1-t)0)\,\le\, t\psi(1)+(1-t)\psi(0)\,=\,t\psi(1)\,,\qquad t\in[0,1]\,.$$  
\end{proof}

Now, we are in position to prove Theorem \ref{tm:WASS1}.

\begin{proof}[Proof of Theorem \ref{tm:WASS1}] Fix $x,y\in\R^d$, $x\neq y$, and let $\process{X}$ and $\process{Y}$ be solutions to \eqref{eq1} starting from $x$ and $y$, respectively. Further, define $\tau:=\inf\{t>0:X_t=Y_t\}$ and 
	\begin{equation*}Z_t\,:=\,\left\{\begin{array}{cc}
	Y_t\,, & t<\tau\,, \\
	X_t\,,& t\ge \tau\,,
	\end{array}\right.\qquad t\ge0\,. \end{equation*}	By employing the strong Markov property it is easy to see that $\mathbb{P}^y(Z_t\in\cdot)=\mathbb{P}^y(Y_t\in\cdot)$ for all $t\ge0$. Consequently,
	\begin{equation*}\W_{f,p}(\delta_x P_t,\delta_y P_t)\,\le\,\left(\mathbb{E}(f(|X_t-Z_t|)^p)\right)^{1/p}\,,\qquad t\ge0\,.\end{equation*}
	Next, since the mapping $t\mapsto |X_t-Z_t|$ is  absolutely continuous on $[0,\tau)$,
	the function $t\mapsto f(|X_t-Z_t|)$ is differentiable a.e. on $[0,\tau)$ and we have that $$\frac{\D}{\D t}f(|X_t-Z_t|)\,=\,\,\frac{f'(|X_t-Z_t|)}{|X_t-Z_t|}\langle  X_t-Z_t,b(X_t)-b(Z_t)\rangle\,,$$ a.e. on $[0,\tau).$
	Now, by assumption, we get $$\frac{\D}{\D t}f(|X_t-Z_t|)\,\leq\, 0\,,$$ a.e. on $[0,\tau),$ which implies that the function $t\mapsto f(|X_t-Z_t|)$ is non-increasing on $[0,\infty)$. 
	Take now $x,y\in\R^d$  such that $0<f(|x-y|)\le \gamma$ (which exist by (iii)). Thus, for such starting points,  $f(|X_t-Z_t|)\le \gamma$ on $[0,\infty).$ Now, by assumption, 
	$$\frac{\D }{\D t}f(|X_t-Z_t|)\,\le\,-\Gamma\,\psi(f(|X_t-Z_t|))\,,$$ a.e. on $[0,\tau),$
	which	together with Lemma \ref{lm:WASS1} gives $$f(|X_t-Z_t|)\,\leq\, \Psi_{f(|x-y|)}^{-1}(\Gamma t)\,,\qquad t\geq 0\,.$$  For $t\ge\tau$ the term on the left-hand side vanishes, and the term on the right-hand side is well defined and strictly positive ($\psi(t)$ is  convex and $\psi(t)=0$ if and only if $t=0$). Now, by taking the expectation and infimum we conclude $$\W_{f,p}(\delta_x P_t,\delta_y P_t)\,\leq\, \Psi_{f(|x-y|)}^{-1}(\Gamma t)\,,\qquad t\geq 0\,,$$ which proves (a).
	
	The relations in (b) now follow from (a) and Lemma \ref{lm:WASS1}.

	Let us  prove (c). If $f(|x-y|)\le \gamma$ for all $x,y\in\R^d$, then the assertion follows from (a). Assume that there are
	$x,y\in\R^{d}$  such that $f(|x-y|)> \gamma$. Observe that, $\delta=0$ if and only if $f(t)\le \gamma$ for all $t\in[0,\infty)$. Thus, $\delta>0$, and we have that
	$$f\left(\frac{|x-y|}{\lceil\delta|x-y|\rceil}\right)\,\le\, f(\delta^{-1})\,\le\, \gamma\,.$$ 
	Take  $z_0,\dots,z_{\lceil\delta|x-y|\rceil}\in\R^{d}$, such that $z_0=x$ and  $$z_{i+1}\,=\,z_i+\frac{y-x}{\lceil\delta|x-y|\rceil}\,,\qquad i=0,\dots,\lceil\delta|x-y|\rceil-1\,.$$ By construction,
	$f(|z_0-z_{1}|)=\dots=f(|z_{\lceil\delta|x-y|\rceil-1}-z_{\lceil\delta|x-y|\rceil}|)\le \gamma$.
	Thus, using (b) we conclude that for $x,y\in\R^{d}$ such that $f(|x-y|)> \gamma$, \begin{align*}
	\W_{f,p}(\delta_x P_t,\delta_y P_t)&\,\leq\, \W_{f,p}(\delta_{z_0} P_t,\delta_{z_{1}} P_t)+\cdots+\W_{f,p}(\delta_{z_{\lceil\delta|x-y|\rceil-1}} P_t,\delta_{z_{\lceil\delta|x-y|\rceil}} P_t)\\
	&\,\le\, \lceil\delta|x-y|\rceil \Psi_\gamma^{-1}(\Gamma t)\,,\qquad t\geq0\,.	\end{align*}  
	Finally, observe that if $t>0$ is such that $f(t)\le \gamma$, then $\delta\le 1/t$, that is, $\delta t\le 1.$ Hence, for  $x,y\in\R^d$ such that $f(|x-y|)\le \gamma$ we have $\lceil\delta|x-y|\rceil=1,$ which concludes the proof.
\end{proof}

Let us now give several remarks.
\begin{remark}\label{rm3}
	\begin{itemize}
		\item[(i)] If the condition in \eqref{eqWASS1} holds for some $\gamma>0$, then it also holds for any $0< \bar \gamma\le \gamma.$
		\item[(ii)] By replacing the  condition in \eqref{eqWASS1} with \begin{align*}&f(|x-y|)^{p-1}f'(|x-y|)\langle x-y,b(x)-b(y)\rangle\\&\,\leq\,\left\{\begin{array}{cc}
		-\frac{\Gamma}{p}|x-y|\psi(f^p(|x-y|))\,, & f^p(|x-y|)\le \gamma\,, \\
		0\,,& f^p(|x-y|)> \gamma\,,
		\end{array}\right.\end{align*} a.e. on $\R^d$ for $\gamma>0$ and $\Gamma>0$, leads to analogous results ($f(t)$ is replaced by $f^p(t)$ in every relation).
		\item[(iii)] For any  $\mu,\nu\in\mathcal{P}$ it holds that
		$$\W_{f,p}(\mu P_t,\nu P_t)\,\leq\, (\delta\, \W_p(\mu,\nu)+1)\Psi_\gamma^{-1}(\Gamma t)\,,\qquad t\geq0\,.$$
		In particular, for $f(t)=t$ we have that
		$$\W_{p}(\mu P_t,\nu P_t)\,\leq\, \left( \frac{\W_p(\mu,\nu)}{\gamma}+1\right)\Psi_\gamma^{-1}(\Gamma t)\,,\qquad t\geq0\,.$$ 
		\item[(iv)]  By taking $\psi(t)=t$ we obtain geometric rate of convergence with $\Psi_{f(|x-y|)}^{-1}(\Gamma t)=f(|x-y|)\E^{-\Gamma t}$. This result can be also obtained in an alternative way (without Lemma \ref{lm:WASS1}, that is, Gronwall's inequality), by applying It\^{o}'s lemma to the processes $\{f(|X_t-Z_t|)\}_{t\ge0}$ and $\{\E^{\Gamma t}f(|X_t-Z_t|)\}_{t\ge0}$.
		\item [(v)] In the case when  $f(t)=\psi(t)=t$, according to \eqref{eqWASS6}, we get \begin{equation}\label{eqWASS13}W_{p}(\mu P_t,\nu P_t)\leq  \W_p(\mu,\nu)\E^{-\Gamma t}\,,\qquad p\ge1\,,\ \mu,\nu\in\mathcal{P}\,,\ t\geq0\,,\end{equation} which is the same results as in \cite{Renesse-Sturm-2005} (for $p=2$). 
		Also, by an analogous approach as in the proof of Theorem \ref{tm:WASS2}, from \eqref{eqWASS13} we see that $\process{X}$ admits a unique invariant $\pi\in\cap_{p\ge1}\mathcal{P}_p$ such that $$\W_{p}(\mu P_t,\pi)\,\leq\,  \W_p(\mu,\pi)\E^{-\Gamma t}\,,\qquad p\ge1\,,\ \mu\in\mathcal{P}_p\,,\ t\geq0\,.$$ 
		\item[(vi)] From \eqref{eqWASS13} we see that  the mapping $\mathcal{P}\ni\mu\mapsto\mu P_t\in\mathcal{P}$ is a contraction for fixed $t>0$, that is, the right-hand side in \eqref{eqWASS13} is strictly smaller than $\W_p(\mu,\nu)$. On the other hand, in the general situation, this is not the case anymore (see (iii)). However, if 
		\begin{equation*}f'(|x-y|)\langle x-y,b(x)-b(y)\rangle\,\leq\,
		-\Gamma\,|x-y|\psi(f(|x-y|))\,,\qquad x,y\in\R^d\,,\end{equation*} then from \eqref{eqWASS2} we have that for all $x,y\in\R^d$ and all $t\ge0$, $$\W_{f,p}(\delta_x P_t,\delta_y P_t)\,\le\,\Psi^{-1}_{f(|x-y|)}(\Gamma t)\,\le\,\Psi^{-1}_{f(|x-y|)}(0)\,=\,f(|x-y|)\,,$$ that is, $$\W_{f,p}(\mu P_t,\nu P_t)\,\le\, \W_{f,p}(\mu,\nu)\,,\qquad p\ge1\,,\ \mu,\nu\in\mathcal{P}\,,\ t\ge0\,.$$ Thus, the mapping $\mathcal{P}\ni\mu\mapsto\mu P_t\in\mathcal{P}$ is contractive for  any fixed $t\ge0.$
	\end{itemize}
\end{remark}

We now prove Theorem \ref{tm:WASS2}.

\begin{proof}[Proof of Theorem \ref{tm:WASS2}] 
	First, we prove  that $\process{X}$ admits an invariant probability measure.
	According to \cite[Theorem 3.1]{Meyn-Tweedie-AdvAP-II-1993}, this will follow if
	we show that for each $x\in\R^d$ and $0<\varepsilon<1$ there is a compact set $C\subset\R^d$ (possibly depending on $x$ and $\varepsilon$) such that $$\liminf_{t\nearrow\infty}\frac{1}{t}\int_0^tp(s,x,C)\,\D s\,\ge\, 1-\varepsilon\,.$$
	By taking $y=0$ in \eqref{eqWASS7} we have that $$\langle x,b(x)\rangle \,\le\, \langle x,b(0)\rangle-\Gamma|x|\psi(|x|)\,\le\, |b(0)||x|-\Gamma|x|\psi(|x|)\,,\qquad x\in\R^d\,.$$ 
	In particular, for $\V(x)=|x|^2$ we have that $$\mathcal{L}\V(x)\,=\,2\langle x,b(x)\rangle+{\rm Tr}\,\sigma\sigma^T\,\le\, {\rm Tr}\,\sigma\sigma^T+2|b(0)||x|-2\Gamma |x|\psi(|x|)\,, \qquad x\in\R^d\,.$$ 
	Now, since every super-additive convex function is necessarily non-decreasing and unbounded,
	we conclude that there is $r_0>0$
	large enough such that 
	$${\rm Tr}\,\sigma\sigma^T+2|b(0)||x|\,\le\, \Gamma |x|\psi(|x|)\,,\qquad |x|\ge r_0\,,$$ that is,
	\begin{align*}
	\mathcal{L}\V(x)\,\le\,& \left({\rm Tr}\,\sigma\sigma^T+2|b(0)||x|-2\Gamma|x|\psi(|x|)\right)\Ind_{B_{r_0}(x)}\\
	&+\left({\rm Tr}\,\sigma\sigma^T+2|b(0)||x|-2\Gamma |x|\psi(|x|)\right)\Ind_{B^c_{r_0}(x)}\\
	&\,\le\, \left({\rm Tr}\,\sigma\sigma^T+2|b(0)||x|-2\Gamma|x|\psi(|x|)\right)1_{B_{r_0}(x)}-\Gamma|x|\psi(|x|)\Ind_{B^c_{r_0}(x)}\\
	&\,\le\,\left({\rm Tr}\,\sigma\sigma^T+2|b(0)|r_0+\Gamma r_0\psi(r_0)\right)\Ind_{B_{r_0}(x)}-\Gamma r_0\psi(r_0)\,,\qquad |x|\ge r_0\,.
	\end{align*}
	Clearly, the above relation holds for all $r\ge r_0$ also. Now, according to 
	\cite[Theorem 1.1]{Meyn-Tweedie-AdvAP-III-1993} we conclude that for each $x\in\R^d$ and $r\ge r_0$ we have $$\liminf_{t\nearrow\infty}\frac{1}{t}\int_0^tp(s,x,\bar B_{r}(0))\,\D s\,\ge\,\frac{\Gamma r\psi(r)}{{\rm Tr}\,\sigma\sigma^T+2|b(0)|r+\Gamma r\psi(r)}\,.$$
	The assertion now follows by choosing  $r$ large enough.

	Let us now show that any invariant $\pi\in\mathcal{P}$ of $\process{X}$ has finite all moments. Fix $p\ge2$ and let $\V_p(x)=|x|^p$. By the same reasoning as above, it is easy to see that there are $r_p>0$, $\Gamma_{p,1}>0$ and $\Gamma_{p,2}>0$ such that $$\mathcal{L}\V_p(x)\,\le\, \Gamma_{p,1}\Ind_{B_{r_p}(0)}(x)-\Gamma_{p,2}|x|^{p-1}\psi(|x|)\,,\qquad x\in\R^d\,.$$ Now, from \cite[Theorem 4.3]{Meyn-Tweedie-AdvAP-III-1993} it follows that  $$\int_{\R^d}|x|^{p-1}\psi(|x|)\pi(\D x)\,\le\,\frac{\Gamma_{p,1}}{\Gamma_{p,2}}$$ for any corresponding invariant  $\pi\in\mathcal{P}$.

	Finally, let us prove that $\process{X}$ admits a unique invariant probability measure which satisfies \eqref{eqWASS8}. Let $\pi,\bar{\pi}\in\mathcal{P}$	be two invariant probability measures of $\process{X}$. Then,  for any $\kappa>0$ and $p\ge1$  Remark \ref{rm3} implies that $$\W_p(\pi,\bar{\pi})\,=\,\W_p(\pi P_t,\bar{\pi}P_t)\,\le\,\left(\frac{\W_p(\pi,\bar{\pi})}{\kappa}+1\right)\Psi_\kappa^{-1}(\Gamma t)\,,\qquad t\ge0\,.$$ Now, by letting $t\to\infty$ we see that $\W_p(\pi,\bar{\pi})=0,$ that is, $\process{X}$ admits a unique invariant  $\pi\in\mathcal{P}.$ Finally, for any $\kappa>0$, $p\ge1$ and $\mu\in\mathcal{P}_p$, by employing Remark \ref{rm3} again, we have that 
	$$\W_p(\pi,\mu P_t)\,=\,\W_p(\pi P_t,\mu P_t)\,\le\,\left(\frac{\W_p(\pi,\mu)}{\kappa}+1\right)\Psi_\kappa^{-1}(\Gamma t)\,,\qquad t\ge0\,,$$ which concludes the proof.	
\end{proof}

Let us now give a simple example satisfying  \eqref{eqWASS1} and \eqref{eqWASS7}.

\begin{example}\label{ex1}{\rm Let $p>1$, $b(x)=-{\rm sgn}(x)|x|^p$, $\sigma(x)\equiv\sigma\in\R$, $f(t)=t$,  $\gamma>0$ and $\psi(t)=t^p$. Now, it is easy to see that
		$b(x)$ cannot satisfy the relation in \eqref{eqWASS6}. On the other hand, an elementary computation shows that there is $\Gamma>0$ such that \eqref{eqWASS1} holds true. Thus, we have \eqref{eqWASS5} with $\delta=\gamma^{-1}$. Also, $\lim_{t\to\infty}\sqrt[p-1]{t}\Psi_\kappa^{-1}(t)=1/\sqrt[p-1]{p-1}$, $\kappa>0$.
		
		Let us also remark that one can show that the  same result holds in the multidimensional case  with  $b(x_1,\dots,x_d)=(-{\rm sgn}(x_1)|x_1|^p,\dots,-{\rm sgn}(x_d)|x_d|^p)$.}
\end{example}

\subsection{Ergodicity of Markov processes with jumps}

Let $\process{Y}$ be a $d$-dimensional L\'evy process with L\'evy triplet $(\beta,\gamma,\nu)$. Further, let $b:\R^d\to\R^d$ be continuous and such that 
\begin{description}
	\item[(J1)]  for any $r>0$ there is $\Gamma_r>0$ such that for all $x,y\in B_r(0)$,  $$\langle x-y,b(x)-b(y)\rangle\,\leq\, \Gamma_r|x-y|^{2}\,;$$
	\item[(J2)] there is $\Gamma>0$ such that for all $x\in\R^d$,  $$\langle x,b(x)\rangle\,\leq\ \Gamma (1+|x|^2)\,.$$
\end{description}
Then, according to \cite[Theorem 1.1, and Lemmas 2.4 and 2.5]{Majka-2016}, the SDE \begin{equation}\label{eq:SDEJ}\D X_t\,=\,b(X_t)\,\D t+\D Y_t\,,\qquad X_0=x\in\R^{d}\,,\end{equation} admits a unique strong non-explosive solution $\process{X}$ which is a strong Markov process and satisfies the $C_b$-Feller property. 
\begin{lemma}\label{lm:jump1} Assume that $\mathbb{E}[|Y_1|^p]<\infty$ (or, equivalently, $\int_{B^c_1(0)}|y|^p\nu(\D y)<\infty$) for some $p>0$.
	Then, there is a constant $\Delta>0$ such that 
	$$\mathbb{E}^{x}\bigl[|X_t|^p\bigr] \,\le\, (|x|^p+1)\E^{\Delta t}\,, \qquad  t \ge 0\,,\  x \in \R^d\,.$$
\end{lemma}
\begin{proof} Let $\chi\in C^2(\R^d)$ be such that $\chi(x)\ge0$, $\chi(x)\le|x|^p$ and $\chi(x)=|x|^p$ for $x\in B_1^c(0)$. Further, for $n\in\N$,  let $\chi_n\in C^2_b(\R^d)$ be such that $\chi_n(x)\ge0$, $\chi_n(x)=\chi|_{B_{n+1}(0)}(x)$ and $\chi_n(x)\to \chi(x)$ as $n\to\infty$, and $\tau_n:=\inf\{t\ge0:X_t\in B^c_n(0)\}$. Then, according to It\^o's formula (see \cite[Remark 2.2]{Albeverio-Brzezniak-Wu-2010}), we have that \begin{align*} \mathbb{E}^x[\chi_n(X_{t\wedge\tau_n})]&\,\le\,\chi_n(x)+\Delta_n(t\wedge\tau_n)+\Delta_n\mathbb{E}^x\left[\int_0^{t\wedge\tau_n}\chi_n(X_{s})\D s\right]\\
	&\;\le\; \chi_n(x)+\Delta_nt+\Delta_n\int_0^{t}\mathbb{E}^x\left[\chi_n(X_{s\wedge\tau_n})\right]\D s\,,\qquad n\in\N\,,\ t\ge0\,,\ x\in\R^d\,,
	\end{align*}
	where the constants $\Delta_n>0$ depend on $p$,  $\beta$, $\gamma$, $b(x)$ and constants $\int_{B_1(0)}|y|^2\nu(\D y)$, $\nu(B_1^c(0))$, $\sup_{x\in B_R(0)}|\nabla\chi_n(x)|$ and $\sup_{x\in B_R(0)}|\nabla^2\chi_n(x)|$, for $R>0$ large enough. Clearly, the functions $\chi_n(x)$ can be chosen such that $\Delta:=\sup_{n\in\N}\Delta_n<\infty.$ Now, since the function $t\mapsto\mathbb{E}^x[\chi_n(X_{t\wedge\tau_n})]$ is bounded and  c\`adl\`ag, Gronwall's lemma implies that \begin{equation*}\mathbb{E}^x[\chi_n(X_{t\wedge\tau_n})]\,\le\, (\chi_n(x)+1) e^{\Delta t} -1\,,\qquad n\in\N\,,\ t\ge0\,,\ x\in\R^d\,.\end{equation*} By letting  $n\to \infty$ monotone convergence theorem and non-explosivity of $\process{X}$  imply that \begin{equation*}\mathbb{E}^x[\chi(X_{t})]\,\le\,(\chi(x)+1) e^{\Delta t} -1\,,\qquad  t\ge0\,,\ x\in\R^d\,.\end{equation*} Finally, we have that \begin{equation*}\mathbb{E}^x\bigl[|X_{t}|^p\bigr]\,\le\,\mathbb{E}^x[\chi(X_{t})]+1\,\le\,(\chi(x)+1)\E^{\Delta t}\,\le\,(|x|^p+1)\E^{\Delta t}\,,\qquad  t\ge0\,,\ x\in\R^d\,.\end{equation*}
\end{proof}
\begin{lemma}\label{lm:jump2} Assume that  $\nu(\R^d)<\infty$.
	Then, the sample paths of $\process{X}$ are piecewise continuous $\Prob^x$-a.s. 
\end{lemma}
\begin{proof}
	Define $\tau_0:=0$ and
	\begin{equation*}
	\tau_n\,:=\,\inf\bigl\{t\ge\tau_{n-1}: |X_t-X_{t-}|>0\bigr\}
	\,=\,\inf\bigl\{t\ge\tau_{n-1}: | Y_t- Y_{t-}|>0\bigr\}\,,\qquad n\ge1\,.
	\end{equation*} Clearly, $\{\tau_n\}_{n\in\N}$ are i.i.d. and
	$\Prob^x(\tau_1>t)=\E^{-\nu(\R^d)t}$ (that is, $\tau_1$ is exponentially distributed with parameter $\nu(\R^d)$) for any $x\in\R^d$. Hence,
	$\process{X}$ is  continuous on $[\tau_n,\tau_{n+1})$, $n\ge0,$ $\Prob^x$-a.s. for all $x\in\R^d$.
\end{proof}
Let now $\process{X}$ be a solution to \eqref{eq:SDEJ} with $b(x)$ satisfying \textbf{(J1)} and \textbf{(J2)}, and with   $\process{Y}$ having finite $p$-th moment, $p\ge1$, and finite L\'evy measure. Then, according to Lemmas \ref{lm:jump1} and \ref{lm:jump2}, if $b(x)$ satisfies  
\eqref{eqWASS1} we conclude that $\process{X}$ satisfies 
\eqref{eqWASS2}, \eqref{eqWASS3}, \eqref{eqWASS4} and \eqref{eqWASS5}.
Further, according to \cite{Albeverio-Brzezniak-Wu-2010} and \cite{Masuda-2007}, for any $f\in C^2(\R^d)$ such that $x\mapsto \int_{B_1^c(0)}f(x+y)\nu(\D y)$ is locally bounded,
$$ f(X_t)-f(X_0)-\int_0^t\mathcal{L}f(X_s)\D s,\qquad t\ge0,$$ is a local $\Prob^x$-martingale, $x\in\R^d$, where \begin{align*}\mathcal{L}f(x)\,=\,&\langle b(x),\nabla f(x)\rangle+ \langle \beta,\nabla f(x)\rangle+\frac{1}{2}{\rm Tr}\,\gamma\,\nabla^2f(x)\\&+\int_{\R^d}\left(f(y+y)-f(x)-\langle y,\nabla f(x)\rangle\Ind_{B_1(0)}(y)\right)\nu(\D y)\,.\end{align*}
\begin{proposition} Let $p\ge1$. Assume that $b(x)$ satisfies \textbf{(J1)}, \textbf{(J2)} and \eqref{eqWASS7}, and that   $\process{Y}$ has finite $p$-th moment and finite L\'evy measure. Then, $\process{X}$ admits a unique invariant $\pi\in\mathcal{P}_p$ such that for any $\kappa>0$, $1\le q\le p$ and $\mu\in \mathcal{P}_q$ it holds that 
	\begin{equation}\label{eq:WASS10}\W_q(\pi,\mu P_t)\,\le\,\left(\frac{\W_q(\pi,\mu)}{\kappa}+1\right)\Psi_\kappa^{-1}(\Gamma t)\,,\qquad t\ge0\,.\end{equation} 
\end{proposition}
\begin{proof} First, observe that 
	\begin{align*}\mathcal{L}f(x)\,=\,&\langle b(x),\nabla f(x)\rangle+ \langle \beta+\int_{B^c_1(0)}y\,\nu(\D y),\nabla f(x)\rangle+\frac{1}{2}{\rm Tr}\,\gamma\,\nabla^2f(x)\\&+\int_{\R^d}\left(f(y+y)-f(x)-\langle y,\nabla f(x)\rangle\right)\nu(\D y)\,.\end{align*} By taking a non-negative $\V_p\in C^2(\R^d)$ such that $\V_p(x)=|x|^p$ on $B_1^c(0)$ from \cite[Lemma 5.1]{Arapostathis-Pang-Sandric-2019} we have that $$\sup_{x\in\R^d}\left|\int_{\R^d}\left(\V_p(y+y)-\V_p(x)-\langle y,\nabla \V_p(x)\rangle\right)\nu(\D y)\right|\,<\,\infty.$$ Now, by  completely the same approach as in the proof of Theorem \ref{tm:WASS2} we conclude that $\process{X}$ admits a unique invariant $\pi\in\mathcal{P}$ such that $\int_{\R^d}|x|^{p-1}\psi(|x|)\pi(\D x)<\infty.$ Thus, $\pi\in\mathcal{P}_p$, and the relation in \eqref{eq:WASS10} follows by the same reasoning as in the proof of Theorem \ref{tm:WASS2}.
\end{proof}

Analogously as in Subsection \ref{du}, in the following proposition  we discuss ergodicity of a class of Markov processes with jumps, obtained through Bochner's subordination method, with respect to Wasserstein distances.

\begin{proposition}\label{p3} 
	Let  $\process{M}$ be a  Markov process with state space $(\R^d,\mathcal{B}(\R^d))$ and semigroup $\process{P}$.  Let $\process{S}$ be a subordinator with characteristic exponent $\phi(u)$, independent of   $\process{M}$. 
	Further, let $\rho$ be a metric on $\R^d$ such that $(\R^d,\rho)$ is a Polish space and  $\mathcal{B}(\R^d_\rho)\subseteq\mathcal{B}(\R^d)$, that is, $\rho$ induces a coarser topology than the standard $d$-dimensional Euclidean metric on $\R^d$. Assume,
	that $\process{M}$ admits an invariant $\pi\in\mathcal{P}$ such that  $\W_{\rho,p}(\delta_x P_t,\pi)\le \Gamma(x)r(t),$ $t\ge0$, $x\in\R^d$, where $r:[0,\infty)\to[1,\infty)$ is Borel measurable and $\Gamma(x)\ge0$. 
	Then,  $\W_{\rho,p}(\delta_x P^\phi_t,\pi)\le \Gamma(x)r_\phi(t),$ $t\ge0$, $x\in\R^d,$ where $r_\phi(t)=\left(\mathbb{E}[r^p(S_t)]\right)^{1/p}.$ 
\end{proposition}

\begin{proof} First, recall that if $\pi$ is an invariant measure for $\process{M}$, then it is also invariant for $\process{M^\phi}$. Next,  \cite[Theorem 4.1]{Villani-Book-2009} implies that for each $s\in[0,\infty)$ there is $\Pi_s\in\mathcal{C}(\delta_x P_s,\pi)$ such that $\W_{\rho,p}(\delta_x P^\phi_s,\pi)=\int_{\R^d\times \R^d}\rho(y,z)\Pi_s(\D y,\D z)$. Now, we have that
	\begin{align*}
	\W^p_{\rho,p}(\delta_xP_t^\phi,\pi)&\,=\,\inf_{\Pi\in\mathcal{C}(\delta_xP_t^\phi,\pi)}\int_{\R^d\times\R^d}\rho^p(y,z)\Pi(\D y,\D z)\\&\,\le\,
	\int_{\R^d\times\R^d}\rho^p(y,z)\int_{[0,\infty)}\Pi_s(\D y,\D z)\mu_t(\D s)\\&\,\le
	\,	\int_{[0,\infty)} \W^p_{\rho,p}(\delta_xP_s,\pi)\mu_t(\D s)\\&\,\le\, \Gamma^p(x)\int_{[0,\infty)} r^p(s)\mu_t(ds)\\
	&\,=\,\Gamma^p(x)\mathbb{E}[r^p(S_t)]\,,\end{align*}  which completes the proof.
\end{proof}

\section*{Acknowledgements}
This research was supported by the Croatian Science Foundation (under Project 8958).
We also thank two anonymous referees for the helpful comments that have led to significant improvements of the results in the article.

\bibliographystyle{alpha}
\bibliography{References}

\end{document}